\documentclass{amsart}
\usepackage{amsmath}
\usepackage{amssymb}
\usepackage{amsthm}
\usepackage{mathrsfs}

\DeclareMathOperator{\supp}{supp}

\DeclareMathOperator{\Imaginary}{Im}
\DeclareMathOperator{\Real}{Re}

\newcommand{\oP}{\overline{P}}
\newcommand{\oQ}{\overline{Q}}

\newcommand{\lv}{\lvert}
\newcommand{\rv}{\rvert}
\newcommand{\lV}{\lVert}
\newcommand{\rV}{\rVert}

\newcommand{\mE}{\mathcal{E}}

\newcommand{\mP}{\mathcal{P}}

\newcommand{\bC}{\mathbb{C}}

\newcommand{\bH}{\mathbb{H}}

\newcommand{\bN}{\mathbb{N}}

\newcommand{\bR}{\mathbb{R}}

\newcommand{\bZ}{\mathbb{Z}}

\def\sideremark#1{\ifvmode\leavevmode\fi\vadjust{\vbox to0pt{\vss
 \hbox to 0pt{\hskip\hsize\hskip1em
 \vbox{\hsize3cm\tiny\raggedright\pretolerance10000
 \noindent #1\hfill}\hss}\vbox to8pt{\vfil}\vss}}}

\newcommand{\comment}[1]{}

\newtheorem{thm}{Theorem}[section]
\newtheorem{prop}[thm]{Proposition}
\newtheorem{lem}[thm]{Lemma}
\newtheorem{cor}[thm]{Corollary}

\theoremstyle{definition}
\newtheorem{defn}[thm]{Definition}

\theoremstyle{remark}
\newtheorem{remark}[thm]{Remark}

\numberwithin{equation}{section}

\newcommand{\Complex}{\mathbb C}

\newcommand{\ddbar}{\overline\partial}
\newcommand{\pr}{\partial}
\newcommand{\ol}{\overline}
\newcommand{\Td}{\widetilde}
\newcommand{\norm}[1]{\left\Vert#1\right\Vert}
\newcommand{\abs}[1]{\left\vert#1\right\vert}

\newcommand{\set}[1]{\left\{#1\right\}}
\newcommand{\To}{\rightarrow}

\usepackage{esint}
\theoremstyle{remark}
\newtheorem{rem}[thm]{Remark}

\begin{document}

\title{Extremal metrics for the ${Q}^\prime$-curvature in three dimensions}
\author{Jeffrey S.\ Case}
  \address{Department of Mathematics, McAllister Building, The Pennsylvania State University, State College, PA 16802}
  \thanks{JSC was partially supported by NSF Grant DMS-1004394}
  \email{jscase@psu.edu}
\author{Chin-Yu Hsiao}
  \address{Institute of Mathematics, Academia Sinica, 6F, Astronomy-Mathematics Building, No.1, Sec.4, Roosevelt Road, Taipei 10617, Taiwan}
  \thanks{CYH was supported by Taiwan Ministry of Science of Technology project 
103-2115-M-001-001, 104-2628-M-001-003-MY2 and the Golden-Jade fellowship of Kenda Foundation}
 \email{chsiao@math.sinica.edu.tw}
\author{Paul Yang}
  \thanks{PY was partially supported by NSF Grant DMS-1104536}
  \address{Department of Mathematics \\ Princeton University \\ Princeton, NJ 08540}
  \email{yang@math.princeton.edu}
\begin{abstract}
We construct contact forms with constant $Q^\prime$-curvature on compact three-dimensional CR manifolds which admit a pseudo-Einstein contact form and satisfy some natural positivity conditions.  These contact forms are obtained by minimizing the CR analogue of the $II$-functional from conformal geometry.  Two crucial steps are to show that the $P^\prime$-operator can be regarded as an elliptic pseudodifferential operator and to compute the leading order terms of the asymptotic expansion of the Green's function for $\sqrt{P^\prime}$.
\end{abstract}
\maketitle

\section{Introduction}
\label{sec:intro}

The geometry of CR manifolds is studied via a choice of contact form and the induced Levi form.  A natural question is whether there are preferred choices of contact form.  One such choice is a CR Yamabe contact form, which has the property that the pseudohermitian scalar curvature is constant.  Such contact forms exist on all compact CR manifolds~\cite{ChengMalchiodiYang2013,Gamara2001,GamaraYacoub2001,JerisonLee1987,JerisonLee1989}.  Another such choice is a pseudo-Einstein contact form with constant $Q^\prime$-curvature~\cite{CaseYang2012,Hirachi2013}.  The primary goal of this article is to show that the latter class of contact forms always exist in dimension three under natural positivity assumptions.

The idea of the $Q^\prime$-curvature arose in the work of Branson, Fontana and Morpurgo~\cite{BransonFontanaMorpurgo2007} on Moser--Trudinger and Beckner--Onofri inequalities on the CR spheres.  On any even-dimensional Riemannian manifold $(M^n,g)$, the critical GJMS operator $P_n$ is a conformally covariant differential operator $P_n$ with leading order term $(-\Delta)^{n/2}$ which controls the behavior of the critical $Q$-curvature $Q_n$ within a conformal class (cf.\ \cite{Branson1995}).  Specializing to the case of the standard $n$-sphere $(S^n,g_0)$ in even dimensions, Beckner~\cite{Beckner1993} and, via different techniques, Chang and the third-named author~\cite{ChangYang1995}, used these objects to establish the Beckner--Onofri inequality:
\begin{equation}
\label{eqn:beckner_onofri}
\int_{S^n} w\,P_nw + 2\int_{S^n}Q_nw - \frac{2}{n}\left(\int_{S^n} Q_n\right)\log\fint_{S^n} e^{nw} \geq 0
\end{equation}
for all $w\in W^{n/2,2}$ and for $Q_n$ an explicit (nonzero) dimensional constant.  Moreover, equality holds in~\eqref{eqn:beckner_onofri} if and only if $e^{2w}g_0$ is Einstein, or equivalently, if and only if $e^{2w}g_0=\Phi^\ast g_0$ for $\Phi$ an element of the conformal transformation group of $S^n$.

Branson, Fontana and Morpurgo investigated to what extent the above discussion holds on the standard CR spheres $(S^{2n+1},T^{1,0}S^{2n+1},\theta_0)$.  While it has long been known that there is a CR covariant operator $P_n$ with leading order term $(-\Delta_b)^{n+1}$, this operator has an infinite-dimensional kernel, namely the space $\mP$ of CR pluriharmonic functions~\cite{Graham1984}.  For this reason one does not expect $P_n$ to give rise to a CR analogue of the sharp Beckner--Onofri inequality~\eqref{eqn:beckner_onofri}.  Instead, Branson, Fontana and Morpurgo~\cite{BransonFontanaMorpurgo2007} observed that there is another operator $P_n^\prime$, defined only on $\mP$, with all of the desired properties.  That is, $P_n^\prime$ has leading term $(-\Delta_b)^{n+1}$, is CR covariant, and there is a (nonzero) dimensional constant $Q_n^\prime$ such that
\begin{equation}
\label{eqn:bfm}
\int_{S^{2n+1}} w\,P_n^\prime w + 2\int_{S^{2n+1}}Q_n^\prime w - \frac{2}{n+1}\left(\int_{S^{2n+1}} Q_n^\prime\right)\log\fint_{S^{2n+1}} e^{(n+1)w} \geq 0
\end{equation}
for all $w\in W^{n+1,2}\cap\mP$.  Moreover, equality holds in~\eqref{eqn:bfm} if and only if $e^{2w}\theta_0$ is pseudo-Einstein and torsion-free, or equivalently, if and only if $e^{2w}\theta_0=\Phi^\ast\theta_0$ for $\Phi$ a CR automorphism of $(S^{2n+1},T^{1,0}S^{2n+1})$.

In light of~\eqref{eqn:beckner_onofri}, it is natural to seek metrics of constant $Q$-curvature within a given conformal class on an even-dimensional Riemannian manifold.  This question has been intensively studied in four dimensions.  In particular, Chang and the third-named author~\cite{ChangYang1995} showed that on any compact Riemannian four-manifold $(M^4,g)$ for which the Paneitz operator $P_4$ is nonnegative with trivial kernel and for which $\int Q_4<16\pi^2$, one can construct a metric $\hat g:=e^{2w}g$ for which $\hat Q_4$ is constant by minimizing the functional
\[ II(w) := \int_M w\,P_4w + 2\int_M Q_4w - \frac{1}{2}\left(\int_M Q_4\right)\log\fint_M e^{4w} . \]
This construction, and various modifications of it, have played an important role in studying the geometry of four-manifolds; see~\cite{Chang2005} for further discussion.

The purpose of this article is to show that one can similarly construct contact forms with constant $Q^\prime$-curvature on a compact three-dimensional CR manifold under natural positivity assumptions.  To explain this, let us first recall the essential features of the $Q^\prime$-curvature~\cite{CaseYang2012}.  On any pseudohermitian three-manifold $(M^3,T^{1,0}M,\theta)$, there is a differential operator $P_4^\prime\colon\mP\to C^\infty(M)$ defined on the space $\mP$ of CR pluriharmonic functions with the properties that $P_4^\prime$ has leading term $\Delta_b^2$, is symmetric in the sense that the pairing $(u,v)\mapsto \int u\,P_4^\prime v$ is symmetric on $\mP$, and satisfies the transformation formula
\begin{equation}
 \label{eqn:intro_pprime_transformation}
 e^{2w}\hat P_4^\prime(u) = P_4^\prime(u) \mod \mP^\perp
\end{equation}
for all $u\in\mP$, where $w\in C^\infty(M)$ and $\hat P_4^\prime$ is defined in terms of $\hat\theta=e^w\theta$.  The analytic properties of the $P^\prime$-operator are improved by projecting onto $\mP$.  As we will see, if $\tau\colon C^\infty(M)\to\mP$ is the orthogonal projection, then the operator $\oP_4^\prime:=\tau P_4^\prime\colon\mP\to\mP$ is a formally self-adjoint elliptic pseudodifferential operator.

In general, one cannot associate an analogue of the $Q$-curvature to $P_4^\prime$.  However, one can do so when restricting to pseudo-Einstein contact forms.  A contact form $\theta$ on $(M^3,T^{1,0}M)$ is pseudo-Einstein if its scalar curvature $R$ and torsion $A_{11}$ satisfy the relation $\nabla_1R=i\nabla^1A_{11}$.  This is equivalent to requiring that $\theta$ is locally volume-normalized with respect to a nonvanishing closed $(2,0)$-form~\cite{Hirachi1990}; such contact forms always exist on boundaries of domains in $\bC^2$~\cite{FeffermanHirachi2003}.  For pseudo-Einstein contact forms, one can define a scalar invariant $Q_4^\prime$ which satisfies a simple transformation rule in terms of $P_4^\prime$ and the CR Paneitz operator $P_4$ upon changing the choice of pseudo-Einstein contact forms.  In particular, $\int Q_4^\prime$ is an invariant of the class of pseudo-Einstein contact forms.  For boundaries of domains, it is a biholomorphic invariant; indeed, it is the Burns--Epstein invariant~\cite{BurnsEpstein1988,CaseYang2012}.

Suppose that $\theta$ is a pseudo-Einstein contact form on $(M^3,T^{1,0}M)$.  Then $\hat\theta=e^w\theta$ is pseudo-Einstein if and only if $w$ is a CR pluriharmonic function~\cite{Hirachi1990}.  In particular, it makes sense to consider the transformation formula for the $Q^\prime$-curvature, and one obtains
\begin{equation}
 \label{eqn:intro_qprime_transformation}
 e^{2w}\hat Q_4^\prime = Q_4^\prime + P_4^\prime(w) \mod \mP^\perp
\end{equation}
(see~\cite{CaseYang2012}).  It is thus natural to consider the scalar quantity $\oQ_4^\prime := \tau Q_4^\prime$.  In particular, on the standard CR three-sphere, $\oP_4^\prime$ is precisely the operator considered by Branson, Fontana and Morpurgo~\cite{BransonFontanaMorpurgo2007} and $\oQ_4^\prime$ is precisely the constant in~\eqref{eqn:bfm}.

We construct contact forms for which $\oQ_4^\prime$ is constant by constructing minimizers of the $II$-functional $II\colon\mP\to\bR$ given by
\begin{equation}
\label{eqn:ii_functional}
II(w) = \int_M w\,\oP_4^\prime w + 2\int_M \oQ_4^\prime w - \left(\int_M \oQ_4^\prime\right)\log\fint_M e^{2w}
\end{equation}
on a pseudo-Einstein three-manifold $(M^3,T^{1,0}M,\theta)$.  Note that, since $II$ is only defined on $\mP$, the projections in~\eqref{eqn:ii_functional} can be removed; i.e.\ we can equivalently define the $II$-functional in terms of $P_4^\prime$ and $Q_4^\prime$.  In general the $II$-functional is not bounded below.  However, under natural positivity conditions it is bounded below and coercive, in which case we can construct the desired minimizers.

\begin{thm}
\label{thm:minimizer}
Let $(M^3,T^{1,0}M,\theta)$ be a compact, embeddable pseudo-Einstein three-manifold such that the $P^\prime$-operator $\oP_4^\prime$ is nonnegative and $\ker \oP_4^\prime=\bR$.  Suppose additionally that
\begin{equation}
\label{eqn:q-bound}
\int_M \oQ_4^\prime\,\theta\wedge d\theta < 16\pi^2 .
\end{equation}
Then there exists a function $w\in\mP$ which minimizes the $II$-functional~\eqref{eqn:ii_functional}.  Moreover, the contact form $\hat\theta := e^w\theta$ is such that $\hat\oQ_4^\prime$ is constant.
\end{thm}

The assumptions of Theorem~\ref{thm:minimizer} can be replaced by the assumptions that the CR Paneitz operator is nonnegative and there exists a pseudo-Einstein contact form with scalar curvature nonnegative but not identically zero.  Note that this last assumption implies that the CR Yamabe constant is positive; it would be interesting to know if these conditions are equivalent.  Chanillo, Chiu and the third-named author proved~\cite{ChanilloChiuYang2010} that these assumptions imply that $(M^3,T^{1,0}M)$ is embeddable.  The first- and third-named authors proved~\cite{CaseYang2012} that these assumptions imply both that $\oP_4^\prime\geq0$ with $\ker\oP_4^\prime=\bR$ and that $\int \oQ_4^\prime\leq 16\pi^2$ with equality if and only if $(M^3,T^{1,0}M)$ is CR equivalent to the standard CR three-sphere.  Branson, Fontana and Morpurgo proved~\cite{BransonFontanaMorpurgo2007} Theorem~\ref{thm:minimizer} on the standard CR three-sphere.  In summary, Theorem~\ref{thm:minimizer} implies the following result.

\begin{cor}
 \label{cor:minizer}
 Let $(M^3,T^{1,0}M,\theta)$ be a compact pseudo-Einstein manifold with nonnegative CR Paneitz operator which admits a pseudo-Einstein contact form with positive scalar curvature.  Then there exists a function $w\in\mP$ which minimizes the $II$-functional~\eqref{eqn:ii_functional}.  Moreover, the contact form $\hat\theta := e^w\theta$ is such that $\hat\oQ_4^\prime$ is constant.
\end{cor}

Note that the assumptions of Theorem~\ref{thm:minimizer} are all CR invariant; in particular, if $(M^3,T^{1,0}M)$ is the boundary of a domain in $\bC^2$, the assumptions are biholomorphic invariants.  Note also that the conclusion that $\hat\oQ_4^\prime$ is constant cannot be strengthened to the conclusion that $\hat Q_4^\prime$ is constant: In Section~\ref{sec:example}, we classify the contact forms on $S^1\times S^2$ with its flat CR structure which have $\oQ_4^\prime$ constant, and observe that $Q_4^\prime$ is nonconstant for all of them.

The proof of Theorem~\ref{thm:minimizer} is analogous to the corresponding result in four-dimensional conformal geometry~\cite{ChangYang1995}, though there are many new difficulties we must overcome.  Since we are minimizing within $\mP$, there is a Lagrange multiplier in the Euler equation for the $II$-functional which lives in the orthogonal complement $\mP^\perp$ to $\mP$.  This is avoided by working with $\oP_4^\prime$.  The greater difficulty is to show that minimizers for the $II$-functional exist in $W^{2,2}\cap\mP$ under the hypotheses of Theorem~\ref{thm:minimizer}.  This is achieved by showing that $\oP_4^\prime$ satisfies a Moser--Trudinger-type inequality with the same constant as on the standard CR three-sphere under the positive assumption on $\oP_4^\prime$ and~\eqref{eqn:q-bound}.

To prove that $\oP_4^\prime$ satisfies the above Moser--Trudinger-type inequality, we study the asymptotics of the Green's function of $\bigl(\oP_4^\prime\bigr)^{1/2}$ in enough detail to apply the general results of Fontana and Morpurgo~\cite{FontanaMorpurgo2011}.  To make this precise, we require some more notation.  Fix $\zeta\in M$ and let $(z,t)$ be CR normal coordinates in a neighborhood of $\zeta$ such that $(z(\zeta),t(\zeta))=(0,0)$.  Define $\rho^4(z,t)=\lv z\rv^4+t^2$.  For $m\in\bR$, let
\[ \mE(\rho^m) = \left\{ g \in C^\infty(M\setminus\{\zeta\}) \colon \lv\pr^p_z\pr^q_{\ol z}\pr^r_t g(z,t)\rv \leq \rho(z,t)^{m-p-q-2r} \text{ near $\zeta$} \right\} . \]
The asymptotics of the Green's function of $\bigl(\oP_4^\prime\bigr)^{1/2}$ are as follows.
\begin{thm}
\label{t-gue140827}
Let $(M^3,T^{1,0}M,\theta)$ be a compact embeddable pseudohermitian manifold such that $P_4^\prime$ is nonnegative. Fix $\zeta\in M$ and let $G_\zeta$ be the Green's function for $\bigl(\oP_4^\prime\bigr)^{1/2}$ with pole at $\zeta$.  Then there is a function $B_\zeta\in C^\infty(M\setminus\{\zeta\})$ such that
\[ B_\zeta - \rho^{-2} \in \mE\left(\rho^{-1-\varepsilon}\right) \]
for all $0<\varepsilon<1$ and
\[ G_\zeta = \tau B_\zeta \tau . \]
\end{thm}

We now outline the main argument used in the proof of Theorem~\ref{t-gue140827}. Fix a point $\zeta\in M$, the Green's function of $\bigl(\oP_4^\prime\bigr)^{1/2}$ at $\zeta$ is given by 
\begin{equation}\label{e-gue150424}
G_\zeta=\bigl(\oP_4^\prime\bigr)^{-\frac{1}{2}}\tau\delta_\zeta\tau.
\end{equation} 
Using standard argument in spectral theory, we observe that
\begin{equation}\label{e-gue150424I}
\bigl(\oP_4^\prime\bigr)^{-\frac{1}{2}}=c\int^\infty_0t^{-\frac{1}{2}}\bigl(\oP_4^\prime+t+\pi\bigl)^{-1}dt
\end{equation}
on $(\ker\oP_4^\prime)^\perp\cap\hat\mP$, where $\hat{\mathcal{P}}$ is the space of $L^2$ CR pluriharmonic functions, $\pi\colon\hat{\mathcal{P}}\To{\rm Ker\,}\oP_4^\prime$ is the orthogonal projection, and $c^{-1}=\int^\infty_0t^{-\frac{1}{2}}(1+t)^{-1}dt$.  Theorem~\ref{t-gue140827} then follows from asymptotic expansions for $t^{-\frac{1}{2}}\bigl(\oP_4^\prime+t+\pi\bigl)^{-1}$. By using Boutet de Monvel--Sj\"ostrand's classical theorem for the Szeg{\H o} kernel~\cite{BoutetSjostrand1976}, we first show that $\oP_4^\prime=\tau E_2$ for $E_2$ a classical elliptic pseudodifferential operator on $M$ of order $2$. This allows us to apply classical theory of pseudodifferential operators to find a pseudodifferential operator $G_t$ of order $-2$ depending continuously on $t$ such that $(E_2+t)G_t=I+F_t$, where $F_t$ is a smoothing operator depending continuously on $t$ and $\abs{F_t(x,y)}_{C^m(M\times M)}\lesssim\frac{1}{1+t}$ for all $m\in\mathbb N$. Roughly speaking, $\tau G_t\tau$ is the leading term of the operator $\bigl(\oP_4^\prime+t+\pi\bigl)^{-1}$.  By carefully studying the principal symbol and $t$-behavior of $G_t$, we can show that $G:=c\int^\infty_0t^{-\frac{1}{2}}\tau G_{t}\tau$ is a smoothing operator of order $2$ with  $G\tau\delta_\zeta\tau=\rho^{-2}\mod\mE\left(\rho^{-1-\varepsilon}\right)$, for every $\varepsilon>0$.

This article is organized as follows.  In Section~\ref{sec:bg} we review some basic concepts from pseudohermitian geometry and the definitions of the $P^\prime$-operator and the $Q^\prime$-curvature.  In Section~\ref{sec:green} we use Theorem~\ref{t-gue140827} to show that $\bigl(\oP_4^\prime\bigr)^{1/2}$ satisfies a sharp Moser--Trudinger-type inequality.  In Section~\ref{sec:minimize} we prove Theorem~\ref{thm:minimizer}.  In Section~\ref{sec:example} we show that there is no pseudo-Einstein contact form on $S^1\times S^2$ for which $Q_4^\prime$ is constant.  The remaining sections are devoted to the proof of Theorem~\ref{t-gue140827}.  In Section~\ref{sec:hprelim} we review some basic concepts about pseudodifferential operators and Fourier integral operators.  In Section~\ref{sec:hkernel} we recall some properties of the orthogonal projection $\tau$ established in~\cite{Hsiao2014}.  In Section~\ref{sec:hsymbol} we establish some properties of the principal symbol of $\tau\Delta_b\tau$.  In Section~\ref{sec:hsqrt} we prove Theorem~\ref{t-gue140827}.

\subsection*{Acknowledgements}

The authors thank Po-Lam Yung for his careful reading of an early version of this article.  They also thank the Academia Sinica in Taipei and Princeton University for warm hospitality and generous support while this work was being completed.

\section{Some pseudohermitian geometry}
\label{sec:bg}

In this section we summarize some important concepts in pseudohermitian geometry as are needed to study the $P^\prime$-operator and the $Q^\prime$-curvature in dimension three.

Let $M^3$ be a smooth, oriented (real) three-dimensional manifold.  A \emph{CR structure} on $M$ is a one-dimensional complex subbundle $T^{1,0}\subset T_{\bC}M:= TM\otimes\bC$ such that $T^{1,0}\cap T^{0,1}=\{0\}$ for $T^{0,1}:=\overline{T^{1,0}}$.  Let $H=\Real T^{1,0}$ and let $J\colon H\to H$ be the almost complex structure defined by $J(V+\bar V)=i(V-\bar V)$.

Let $\theta$ be a \emph{contact form} for $(M^3,T^{1,0}M)$; i.e.\ $\theta$ is a nonvanishing real one-form such that $\ker\theta=H$.  Since $M$ is oriented, a contact form always exists, and is determined up to multiplication by a positive real-valued smooth function.  We say that $(M^3,T^{1,0}M)$ is \emph{strictly pseudoconvex} if the \emph{Levi form} $d\theta(\cdot,J\cdot)$ on $H\otimes H$ is positive definite for some, and hence any, choice of contact form $\theta$.  We shall always assume that our CR manifolds are strictly pseudoconvex.

A \emph{pseudohermitian manifold} is a triple $(M^3,T^{1,0}M,\theta)$ consisting of a CR manifold and a contact form.  The \emph{Reeb vector field} $T$ is the vector field such that $\theta(T)=1$ and $d\theta(T,\cdot)=0$.  A \emph{$(1,0)$-form} is a section of $T_{\bC}^\ast M$ which annihilates $T^{0,1}$.  An \emph{admissible coframe} is a nonvanishing $(1,0)$-form $\theta^1$ in an open set $U\subset M$ such that $\theta^1(T)=0$.  Let $\theta^{\bar 1}:=\overline{\theta^1}$ be its conjugate.  Then $d\theta=ih_{1\bar 1}\theta^1\wedge\theta^{\bar 1}$ for some positive function $h_{1\bar 1}$.  The function $h_{1\bar 1}$ is equivalent to the Levi form.

The \emph{connection form $\omega_1{}^1$} and the \emph{torsion form} $\tau_1=A_{11}\theta^1$ determined by an admissible coframe $\theta^1$ are uniquely determined by
\begin{align*}
d\theta^1 & = \theta^1\wedge\omega_1{}^1 + \theta\wedge\tau^1, \\
\omega_{1\bar 1} + \omega_{\bar 11} & = dh_{1\bar 1},
\end{align*}
where we use $h_{1\bar 1}$ to raise and lower indices as normal; e.g.\ $\tau^1=h^{1\bar 1}\tau_{\bar 1}$ for $h^{1\bar 1}=\left(h_{1\bar 1}\right)^{-1}$.  The connection forms determine the \emph{pseudohermitian connection $\nabla$} by
\[ \nabla Z_1 := \omega_1{}^1\otimes Z_1 \]
for $\{Z_1,Z_{\bar 1},T\}$ the dual basis to $\{\theta^1,\theta^{\bar 1},\theta\}$.  The \emph{scalar curvature} $R$ of $\theta$ is given by the expression
 \[ d\omega_1{}^1 = R\theta^1\wedge\theta^{\bar 1} \mod\theta . \]
A (real-valued) function $w\in C^\infty(M)$ is \emph{CR pluriharmonic} if locally $w=\Real f$ for some (complex-valued) function $f\in C^\infty(M,\bC)$ satisfying $Z_{\bar 1}f=0$.  Equivalently, $w$ is a CR pluriharmonic function if
\[ \nabla_1\nabla_1\nabla^1 w + iA_{11}\nabla^1 w = 0 \]
for $\nabla_1:=\nabla_{Z_1}$ (cf.\ \cite{Lee1988}).  We denote by $\mP$ the space of all CR pluriharmonic functions.

Take $\theta\wedge d\theta$ to be the volume form on $M$.  This induces a natural inner product $(\cdot,\cdot)$ on $C^\infty(M)$.  Let $L^2(M)$ and $\hat\mP$ denote the completions of $C^\infty(M)$ and $\mP$, respectively, with respect to this inner product.

The \emph{Paneitz operator} $P_4$ is the differential operator
\begin{align*}
P_4(w) & := 4\nabla^1\left(\nabla_1\nabla_1\nabla^1 w + iA_{11}\nabla^1 w\right) \\
& = \Delta_b^2w + T^2 - 4\Imaginary\nabla^1\left(A_{11}\nabla^1f\right)
\end{align*}
for $\Delta_b:=\nabla^1\nabla_1+\nabla^{\bar 1}\nabla_{\bar 1}$ the \emph{sublaplacian}.  Note in particular that $\mP\subset\ker P_4$.  A key property of the Paneitz operator is that it is CR covariant; if $\hat\theta=e^w\theta$, then $e^{2w}\hat P_4=P_4$ (cf.\ \cite{Hirachi1990}).

\begin{defn}
\label{defn:pprime}
Let $(M^3,T^{1,0}M,\theta)$ be a pseudohermitian manifold.  The \emph{$P^\prime$-operator} $P^\prime\colon\mP\to C^\infty(M)$ is defined by
\begin{equation}
\label{eqn:pprime}
\begin{split}
P_4^\prime f & = 4\Delta_b^2 f - 8\Imaginary\left(\nabla^\alpha(A_{\alpha\beta}\nabla^\beta f)\right) - 4\Real\left(\nabla^\alpha(R\nabla_\alpha f)\right) \\
& \quad + \frac{8}{3}\Real W_\alpha\nabla^\alpha f - \frac{4}{3}f\nabla^\alpha W_\alpha
\end{split}
\end{equation}
for $f\in\mP$, where $W_\alpha:=\nabla_\alpha R - i\nabla^\beta A_{\alpha\beta}$.
\end{defn}

In particular,
\begin{equation}\label{e-gue140811IV}
\begin{split}
&P_4^\prime f=4\Delta^2_bf+R\Delta_bf+\Delta_bRf+(L_1L_2+\ol L_1\ol L_2)f+(L_3+\ol L_3)f+rf,\\
&L_1, L_2, L_3\in C^\infty(M,T^{1,0}M),\ \ r\in C^\infty(M),\ \ f\in\mathcal{P}.
\end{split}
\end{equation}

A key property of the $P^\prime$-operator is its conformal covariance: Let $(M^3,T^{1,0}M,\theta)$ be a pseudohermitian manifold, let $w\in C^\infty(M)$, and set $\hat\theta=e^w\theta$.  Then
\begin{equation}
 \label{eqn:pprime_transformation}
 e^{2w}\hat P_4^\prime(u) = P_4^\prime(u) + P_4\left(uw\right)
\end{equation}
for all $u\in\mP$.  In particular, since $P_4$ is self-adjoint and annihilates CR pluriharmonic functions, \eqref{eqn:pprime_transformation} implies that the $P^\prime$-operator is conformally covariant, mod $\mP^\perp$.

A pseudohermitian manifold $(M^3,T^{1,0}M,\theta)$ is \emph{pseudo-Einstein} if $W_\alpha=0$ for $W_\alpha$ as in Definition~\ref{defn:pprime}.

\begin{defn}
 \label{defn:qprime}
 Let $(M^3,T^{1,0}M,\theta)$ be a pseudo-Einstein manifold.  The \emph{$Q^\prime$-curvature} is
 \begin{equation}
  \label{eqn:qprime}
  Q_4^\prime = 2\Delta_b R - 4\lv A\rv^2 + R^2 .
 \end{equation}
\end{defn}

A key property of the $Q^\prime$-curvature is its conformal covariance: Let $(M^3,T^{1,0}M,\theta)$ be a pseudo-Einstein manifold, let $w\in\mP$, and set $\hat\theta=e^w\theta$.  Hence $\hat\theta$ is pseudo-Einstein~\cite{Hirachi1990}.  Then
\begin{equation}
 \label{eqn:qprime_transformation}
 e^{2w}\hat Q_4^\prime = Q_4^\prime + P_4^\prime(w) + \frac{1}{2}P_4\left(w^2\right) .
\end{equation}
In particular, $Q_4^\prime$ behaves as the $Q$-curvature for $P_4^\prime$, mod $\mP^\perp$.

\section{The Moser--Trudinger inequality for the $P^\prime$-operator}
\label{sec:green}

A key step in our proof of Theorem~\ref{thm:minimizer} is to show that the $P^\prime$-operator satisfies the same sharp Moser--Trudinger-type inequality as its counterpart on the sphere.  This follows from the asymptotic expansion for the Green's function of $\bigl(\oP_4^\prime\bigr)^{1/2}$ given in Theorem~\ref{t-gue140827} and the general Adams-type theorem of Fontana and Morpurgo~\cite{FontanaMorpurgo2011}.

Given $k\in\bN$ and $q>0$, let $W^{k,q}$ denote the non-isotropic Sobolev space, given by the set of all functions $u$ such that $Z_1Z_2\cdots Z_ju\in L^q(M)$ for all $Z_j\in C^\infty(M,T^{1,0}M\oplus T^{0,1}M))$, $j=0,1,2,\dots,k$, 

\begin{thm}
\label{thm:cr_paneitz_adams}
Let $(M^3,T^{1,0}M,\theta)$ be a compact pseudo-Einstein three-manifold for which the $P^\prime$-operator is nonnegative with trivial kernel.  Then there exists a constant $C$ such that
\begin{equation}
\label{eqn:cr_paneitz_adams}
\log\fint_M e^{2(w-w_0)} \leq C + \frac{1}{16\pi^2}\int_M w\,\oP_4^\prime w
\end{equation}
for all $w\in W^{2,2}\cap\mP$.
\end{thm}

\begin{proof}

From Theorem~\ref{t-gue140827} we see that the leading order term of the Green's function for $\oP_4^\prime$ is independent of $(M^3,T^{1,0}M,\theta)$; in particular, it has exactly the same leading order term as the Green's function for the $P^\prime$-operator on the standard CR three-sphere.  Furthermore, the next term in the asymptotic expansion of the Green's function involves a definite loss of power in the asymptotic coordination $\rho$.  Thus, by arguing analogously to the proof of~\cite[Theorem~2.1]{BransonFontanaMorpurgo2007}, we may apply the main result~\cite[Theorem~1]{FontanaMorpurgo2011} to conclude that there is a constant $C>0$ such that
\begin{equation}
\label{eqn:cr_paneitz_adams_unlinear}
\int_M \exp\left(16\pi^2\frac{\left(w-w_0\right)^2}{\int w\oP_4^\prime w}\right)\theta\wedge d\theta \leq C
\end{equation}
for all $f\in W^{2,2}\cap\mP$.  The desired inequality~\eqref{eqn:cr_paneitz_adams} is an immediate consequence of~\eqref{eqn:cr_paneitz_adams_unlinear} and the elementary estimate
\[ 0 \leq 16\pi^2\frac{\left(w-w_0\right)^2}{\int w\oP_4^\prime w} - 2\left(w-w_0\right) + \frac{1}{16\pi^2}\int_M w\oP_4^\prime w . \qedhere \]
\end{proof}

\begin{remark}
A few comments are in order to explain the above constants.  The convention used in~\cite{BransonFontanaMorpurgo2007} is that the sublaplacian is given by $-\Real\nabla^\gamma\nabla_\gamma$, which shows that our definition is $-2$ times theirs.  With this in mind, their formula~\cite[(1.30)]{BransonFontanaMorpurgo2007} for the $P^\prime$-operator shows that our definition is $4$ times theirs.  Finally, they integrate with respect to the Riemannian volume element on $S^3$, regarded as the unit ball in $\bR^4$, while we integrate with respect to $\theta\wedge d\theta$ for $\theta=\Imaginary\overline{\partial}\left(\lv z\rv^2-1\right)$; in particular, our volume form is $2$ times theirs.  Together, these normalizations account for the apparent difference between our constant in~\eqref{eqn:cr_paneitz_adams_unlinear} and the constant appearing in~\cite[(2.11)]{BransonFontanaMorpurgo2007}.  Note that $\theta$ has scalar curvature $R=2$, and hence $\oQ_4^\prime=4$.
\end{remark}

\section{Minimizing the functional $II$}
\label{sec:minimize}

Assuming the results of Section~\ref{sec:hsqrt}, we prove that smooth minimizers of the $II$-functional exist under natural positivity assumptions.  We first construct weak minimizers.

\begin{thm}
\label{thm:existence}
Let $(M^3,T^{1,0}M,\theta)$ be a compact pseudo-Einstein three-manifold such that $\int \oQ_4^\prime<16\pi^2$.  Suppose additionally that the $P_4^\prime$-operator is nonnegative with $\ker \oP_4^\prime=\bR$.  Then
\[ \inf_{w\in W^{2,2}\cap\mP} II[w] \]
is obtained by some function $w\in W^{2,2}\cap\mP$.
\end{thm}

\begin{proof}

Denote $k=\int \oQ_4^\prime$.  Recall that
\[ II[w] = (\oP_4^\prime w,w) + 2\int_M \oQ_4^\prime (w-w_0) - k\log\fint_M e^{2(w-w_0)} \]
for $w_0=\fint w$ the average value of $M$.  If $k\leq 0$, it follows immediately that
\[ II[w] \geq \left(\oP_4^\prime w,w\right) + 2\int_M \oQ_4^\prime (w-w_0) , \]
while if $k>0$, Theorem~\ref{thm:cr_paneitz_adams} implies that
\[ II[w] \geq \left(1-\frac{k}{16\pi^2}\right)\left(\oP_4^\prime w,w\right) + 2\int_M \oQ_4^\prime (w-w_0) - kC . \]
Together, these estimates imply that
\begin{equation}
\label{eqn:II_estimate}
II(w) \geq \left(1-\frac{k^+}{16\pi^2}\right)\left(\oP_4^\prime w,w\right) + 2\int_M \oQ_4^\prime (w-w_0) - C
\end{equation}
for $k^+=\max\{0,k\}$ and $C$ a positive constant depending only on $(M^3,T^{1,0}M,\theta)$.

Denote by $\lambda_1=\lambda_1(\oP_4^\prime)$ the first nonzero eigenvalue
\[ \lambda_1(\oP_4^\prime) = \inf\left\{\frac{(P_4^\prime w,w)}{\lV w\rV_2^2} \colon w\in W^{2,2}\cap\mP,\quad\int_M w=0 \right\} \]
of $\oP_4^\prime$.  By assumption, $\lambda_1>0$.  Together with~\eqref{eqn:II_estimate}, this shows that there are positive constants $c_1,c_2$ depending only on $(M^3,T^{1,0}M,\theta)$ such that
\begin{equation}
\label{eqn:qestimate}
II[w] \geq c_1\lV w-w_0\rV_2^2 - c_2 .
\end{equation}
In particular, $II$ is bounded below.

Let $\{w_k\}\subset\mP$ be a minimizing sequence of $II$, normalized so that $\lV w_k\rV_2=1$ for all $k\in\bN$.  Using~\eqref{eqn:II_estimate} and the local formula~\eqref{eqn:pprime} for $P_4^\prime$, it is easily seen that there is a positive constant $c_3$ depending only on $(M^3,T^{1,0}M,\theta)$ such that
\begin{equation}
\label{eqn:gammaqestimate}
\begin{split}
\left(1-\frac{k^+}{16\pi^2}\right)\int_M(\Delta_bw_k)^2 & \leq c_3\left|\int_M R\lv\nabla_bw_k\rv^2\right| + c_3\left|\int_M\Imaginary A_{\alpha\beta}\nabla^\alpha w_k\nabla^\beta w_k\right| \\
& \quad + 2\left|\int_M Q_4^\prime \left(w_k-(w_k)_0\right)\right| + c_3 .
\end{split}
\end{equation}
On the other hand, given any $\varepsilon>0$, it holds that
\[ \int_M\lv\nabla_bw_k\rv^2 = -\int_M w_k\Delta_bw_k \leq \varepsilon\int_M(\Delta_bw_k)^2 + \frac{1}{4\varepsilon}\lV w_k-(w_k)_0\rV_2^2 . \]
We may thus combine~\eqref{eqn:qestimate} and~\eqref{eqn:gammaqestimate} to conclude that $\{w_k\}$ is uniformly bounded in $W^{2,2}\cap\mP$.  Thus, by choosing a subsequence if necessary, we see that $w_k$ converges weakly in $W^{2,2}\cap\mP$ to a minimizer $w\in W^{2,2}\cap\mP$ of $II$.
\end{proof}

We next show that weak critical points of the $II$-functional are smooth.

\begin{thm}
\label{thm:regularity}
Let $(M^3,T^{1,0}M,\theta)$ be a compact three-dimensional pseudo-Einstein manifold.  Suppose that $w\in W^{2,2}\cap\mP$ is a critical point of the $II$-functional.  Then $w$ is smooth, and moreover, the contact form $\hat\theta:=e^w\theta$ is such that $\hat\oQ_4^\prime$ is constant.
\end{thm}

\begin{proof}

It is readily seen that $w$ is a critical point of the $II$-functional if and only if $w$ is a weak solution to
\begin{equation}
\label{eqn:ii_euler}
P_4^\prime w + Q_4^\prime = \lambda e^{2w}\mod\mP^\perp .
\end{equation}
In particular, if $w$ is smooth, then~\eqref{eqn:qprime_transformation} implies that $\hat\oQ_4^\prime$ is constant. Now, we prove that $w$ is smooth. Fix $\ell\in\mathbb N$ sufficiently large and let $B_\ell$ and $C_\ell$ be as in Theorem~\ref{t-gue150107}. From \eqref{eqn:ii_euler}, we have 
\begin{equation}\label{e-gue150110b}
\tau B_\ell\tau(\lambda e^{2w})=\tau B_\ell\oP^\prime_4w+\tau B_\ell\tau Q_4^\prime=w+\tau C_\ell w+\tau B_\ell\tau Q^\prime_4.
\end{equation}
Note that 
\begin{equation}\label{e-gue150110bI}
\tau C_\ell w+\tau B_\ell\tau Q^\prime_4\in C^\ell(M).
\end{equation}
Since $w\in W^{2,2}$, we have $\Delta_bw\in L^2(M)$. From Theorem~\ref{t-gue150109}, we conclude that 
\[e^{c\abs{w}^2}\in L^1(M),\ \ c>0,\]
and hence 
\begin{equation}\label{e-gue150110bII}
\lambda e^{2w}\in L^q(M),\ \ \forall q>1.
\end{equation}
Since $\tau B_\ell\tau$ is a smoothing operator of order $4-\varepsilon$ for all $0<\varepsilon<1$, 
it holds that (see~\cite[Proposition~2.7]{HsiaoYung2014}) 
\begin{equation}\label{e-gue150110bIII}
\tau B_\ell\tau:W^{k,q}\To W^{k+1,q}, \quad\text{for all $q>1$ and all $k\in\bN_0$} .
\end{equation}
From \eqref{e-gue150110b}, \eqref{e-gue150110bII} and \eqref{e-gue150110bIII}, we obtain that 
\begin{equation}\label{e-gue150110ab}
w\in W^{1,q}, \quad\text{for all $q>1$} .
\end{equation}
From~\eqref{e-gue150110bII} and~\eqref{e-gue150110ab} it is easy to see that $\lambda e^{2w}\in W^{1,q}$ for all $q>1$. From this, \eqref{e-gue150110b} and~\eqref{e-gue150110bIII} we conclude that $w+\tau C_\ell w+\tau B_\ell\tau Q^\prime_4\in W^{2,q}$ for all $q>1$. Continuing in this way, we deduce that $w+\tau C_\ell w+\tau B_\ell\tau Q^\prime_4\in W^{k,q}$ for all $q>1$ and all $k\in\bN_0$ with $k\leq\ell$. Thus, $w\in W^{\ell,q}$, for all $q>1$. Since $\ell$ is arbitrary, we deduce that $w$ is smooth. 
\end{proof}

\begin{proof}[Proof of Theorem~\ref{thm:minimizer}]

By Theorem~\ref{thm:existence}, there is a minimizer $w\in W^{2,2}\cap\mP$ of the $II$-functional.  By Theorem~\ref{thm:regularity}, $w$ is smooth and the contact form $\hat\theta:=e^w\theta$ is such that $\hat\oQ_4^\prime$ is constant, as desired.
\end{proof}

\section{An example}
\label{sec:example}

Here we provide an example to show that minimizers of the $II$-functional, while they have $\oQ_4^\prime$ constant, need not have $Q_4^\prime$-constant.  More precisely, we will prove the following theorem.

\begin{thm}
\label{thm:s1s2_example}
Let $\Gamma$ be a nontrivial dilation of the Heisenberg group $\bH^1$ which fixes the origin $0\in\bH^1$.  Then $S^1\times S^2=(\bH^1\setminus\{0\})/\Gamma$ with its standard CR structure is such that the minimizer of the $II$-functional is unique up to an additive constant, and moreover, the corresponding contact form $\hat\theta$ has $\hat\oQ_4^\prime\equiv0$ but $\hat Q_4^\prime\not\equiv0$.
\end{thm}

\begin{proof}

Let $\rho(z,t)=\left(\lv z\rv^4+t^2\right)^{1/4}$ be the usual pseudo-distance on $\bH^1$.  It is straightforward to check that the contact form $\theta_1=\rho^{-4}\theta_0$ on $\bH^1\setminus\{0\}$ is such that $\theta_1=\Phi^\ast\theta_0$ for
$\Phi(z,t)$ the CR inversion through the pseudo-sphere $\rho^{-1}(1)$.  In particular, $\theta_1$ is flat, and hence $\log\rho\in\mP$.  From~\eqref{eqn:qprime_transformation} it follows that
\begin{equation}
\label{eqn:qprime_inversion}
P_4^\prime\log\rho^{-4} + \frac{1}{2}P_4\log^2\rho^{-4} = 0 .
\end{equation}

Consider now the contact form $\theta:=\rho^{-2}\theta_0$.  It is clear that $\theta$ is invariant under the action of $\Gamma$, and hence $\theta$ descends to a well-defined contact form on $S^1\times S^2$.  Since $\log\rho\in\mP$, we know that $\theta$ is pseudo-Einstein.  From~\eqref{eqn:qprime_transformation} we see that the $Q^\prime$-curvature $Q_4^\prime$ of $\theta$ is
\begin{equation}
\label{eqn:qprime_s1s2}
\rho^{-4}Q_4^\prime = P_4^\prime\log\rho^{-2} + \frac{1}{2}P_4\log^2\rho^{-2} = -\frac{1}{2}P_4\log^2\rho^{-2} ,
\end{equation}
where the second equality uses~\eqref{eqn:qprime_inversion}.  Since the Paneitz operator $P_4$ is self-adjoint and $\mP\subset{\rm Ker\,}P_4$, it follows that $Q_4^\prime$ is orthogonal, with respect to $\theta\wedge d\theta$, to the CR pluriharmonic functions.  In particular, $\oQ_4^\prime\equiv0$.  Furthermore, one can compute directly from~\eqref{eqn:qprime_s1s2} that
\[ Q_4^\prime = 8\frac{\lv z\rv^4-t^2}{\lv z\rv^4+t^2} , \]
which is clearly not identically zero.

Finally, using Lee's formula for the change of the scalar curvature under a conformal change of contact form~\cite[Lemma~2.4]{Lee1988}, we compute that the scalar curvature $R$ of $\theta$ is
\[ R = 2\frac{\lv z\rv^2}{\rho^2} . \]
Since this is nonnegative and $\theta$ is pseudo-Einstein, $\oP_4^\prime$ is nonnegative with trivial kernel~\cite[Proposition~4.9]{CaseYang2012}.  Now, if $\hat\theta=e^u\theta$ is a pseudo-Einstein contact form on $S^1\times S^2$ for which $\hat\oQ_4^\prime\equiv0$, the transformation formula~\eqref{eqn:qprime_transformation} implies that $\oP_4^\prime u\equiv0$, whence $u$ is constant, as desired.
\end{proof}

\section{Preliminaries for pseudodifferential operators}\label{sec:hprelim}

We shall use the following notations: $\bR$ is the set of real numbers, $\bR_+:=\set{x\in\bR;\, x>0}$, $\ol\bR_+:=\set{x\in\bR;\, x\geq0}$, $\mathbb N=\set{1,2,\ldots}$, and $\mathbb N_0=\mathbb N\cup\set{0}$. An element $\alpha=(\alpha_1,\ldots,\alpha_n)$ of $\mathbb N_0^n$ is a \emph{multi-index}, the \emph{size} of $\alpha$ is $\abs{\alpha}=\alpha_1+\cdots+\alpha_n$, and the \emph{length} of $\alpha$ is $l(\alpha)=n$. For $m\in\mathbb N$, we write $\alpha\in\set{1,\ldots,m}^n$ if $\alpha_j\in\set{1,\ldots,m}$ for all $j=1,\ldots,n$. We say that $\alpha$ is strictly increasing if $\alpha_1<\alpha_2<\cdots<\alpha_n$.

Given a multi-index $\alpha$, we write $x^\alpha=x_1^{\alpha_1}\cdots x^{\alpha_n}_n$ for $x=(x_1,\ldots,x_n)$; we write $\pr^\alpha_x=\pr^{\alpha_1}_{x_1}\cdots\pr^{\alpha_n}_{x_n}$ for $\pr_{x_j}=\frac{\pr}{\pr x_j}$ and $\pr^\alpha_x=\frac{\pr^{\abs{\alpha}}}{\pr x^\alpha}$; we write $D^\alpha_x=D^{\alpha_1}_{x_1}\cdots D^{\alpha_n}_{x_n}$ for $D_x=\frac{1}{i}\pr_x$ and $D_{x_j}=\frac{1}{i}\pr_{x_j}$.

Let $z=(z_1,\ldots,z_n)$, $z_j=x_{2j-1}+ix_{2j}$, $j=1,\ldots,n$, be coordinates of $\Complex^n$.  
Given a multi-index $\alpha$, we write $z^\alpha=z_1^{\alpha_1}\cdots z^{\alpha_n}_n$ and $\ol z^\alpha=\ol z_1^{\alpha_1}\cdots\ol z^{\alpha_n}_n$; we write $\frac{\pr^{\abs{\alpha}}}{\pr z^\alpha}=\pr^\alpha_z=\pr^{\alpha_1}_{z_1}\cdots\pr^{\alpha_n}_{z_n}$, where $\pr_{z_j}=
\frac{\pr}{\pr z_j}=\frac{1}{2}(\frac{\pr}{\pr x_{2j-1}}-i\frac{\pr}{\pr x_{2j}})$ for all $j=1,\ldots,n$; similarly, we write 
$\frac{\pr^{\abs{\alpha}}}{\pr\ol z^\alpha}=\pr^\alpha_{\ol z}=\pr^{\alpha_1}_{\ol z_1}\cdots\pr^{\alpha_n}_{\ol z_n}$, where $\pr_{\ol z_j}=
\frac{\pr}{\pr\ol z_j}=\frac{1}{2}(\frac{\pr}{\pr x_{2j-1}}+i\frac{\pr}{\pr x_{2j}})$ for all $j=1,\ldots,n$.

Let $M$ be a smooth manifold. We denote by $\langle\,\cdot\,,\cdot\,\rangle$ the pointwise duality between $TM$ and $T^*M$.
We extend $\langle\,\cdot\,,\cdot\,\rangle$ bilinearly to $T_{\bC}M\times T_{\bC}^*M$.
Let $E$ be a $C^\infty$ vector bundle over $M$. The fiber of $E$ at $x\in M$ are denoted by $E_x$. Let $Y\subset M$ be an open set. The spaces of
smooth sections of $E$ over $Y$ and distributional sections of $E$ over $Y$ are denoted by $C^\infty(Y, E)$ and $\mathscr D'(Y, E)$, respectively.
Let $\mathscr E'(Y, E)$ be the subspace of $\mathscr D'(Y, E)$ whose elements have compact support in $Y$.
For $m\in\bR$, let $H^m(Y, E)$ denote the Sobolev space
of order $m$ of sections of $E$ over $Y$. Put
\begin{align*}
H^m_{\rm loc\,}(Y, E) & =\big\{u\in\mathscr D'(Y, E) \colon \varphi u\in H^m(Y, E) \text{ for all } \varphi\in C^\infty_0(Y)\big\}\,,\\
      H^m_{\rm comp\,}(Y, E) & =H^m_{\rm loc}(Y, E)\cap\mathscr E'(Y, E)\,.
\end{align*} 

Fix a smooth density of integration on $M$. If
$A\colon C^\infty_0(M,E)\To \mathscr D'(M,F)$
is continuous, we write $A(x, y)$ to denote the distributional kernel of $A$.
The following two statements are equivalent:
\begin{enumerate}
\item[(a)] $A$ is continuous as a mapping from $\mathscr E'(M,E)$ to $C^\infty(M,F)$.
\item[(b)] $A(x,y)\in C^\infty(M\times M,E_y\boxtimes F_x)$.
\end{enumerate}
If $A$ satisfies (a) or (b), we say that $A$ is smoothing. Let
$B\colon C^\infty_0(M,E)\to \mathscr D'(M,F)$ be a continuous operator. 
We write $A\equiv B$ if $A-B$ is a smoothing operator.

Let $H(x,y)\in\mathscr D'(M\times M,E_y\boxtimes F_x)$. We also denote by $h$ the unique continuous operator $H\colon C^\infty_0(M,E)\to\mathscr D'(M,F)$ with distribution kernel $H(x,y)$. We henceforth identify $H$ with $H(x,y)$. 

Recall the H\"ormander symbol spaces:

\begin{defn}\label{d-gue140329}
Let $M\subset\bR^N$ be an open set and let $m\in\bR$. $S^m_{1,0}(M\times\bR^{N_1})$ is the space of all $a\in C^\infty(M\times\bR^{N_1})$ such that for all compact $K\Subset M$ and all $\alpha\in\mathbb N^N_0$, $\beta\in\mathbb N^{N_1}_0$, there is a constant $C>0$ such that 
\[\abs{\pr^\alpha_x\pr^\beta_\theta a(x,\theta)}\leq C(1+\abs{\theta})^{m-\abs{\beta}} \quad\text{for all}\quad (x,\theta)\in K\times\bR^{N_1}.\]
Denote 
\[S^{-\infty}(M\times\bR^{N_1}):=\bigcap_{m\in\bR}S^m_{1,0}(M\times\bR^{N_1}).\]
Let $a_j\in S^{m_j}_{1,0}(M\times\bR^{N_1})$ for $j\in\bN_0$ with $m_j\to-\infty$ as $j\to\infty$. Then there exists $a\in S^{m_0}_{1,0}(M\times\bR^{N_1})$, unique modulo $S^{-\infty}(M\times\bR^{N_1})$, such that $a-\sum\limits^{k-1}_{j=0}a_j\in S^{m_k}_{1,0}(M\times\bR^{N_1})$ for all $k\in\{0,1,2,\dotsc\}$. 

If $a$ and $a_j$ have the properties above, we write $a\sim\sum^{\infty}_{j=0}a_j$ in $S^{m_0}_{1,0}(M\times\bR^{N_1})$. 

Let $S^m_{{\rm cl\,}}(M\times\bR^{N_1})$ be the space of all symbols $a(x,\theta)\in S^m_{1,0}(M\times\bR^{N_1})$ with 
\[\mbox{$a(x,\theta)\sim\sum\limits^\infty_{j=0}a_{m-j}(x,\theta)$ in $S^m_{1,0}(M\times\bR^{N_1})$},\]
with $a_k(x,\theta)\in C^\infty(M\times\bR^{N_1})$ positively homogeneous of degree $k$ in $\theta$; that is, $a_k(x,\lambda\theta)=\lambda^ka_k(x,\theta)$ for all $\lambda\geq1$ and all $\abs{\theta}\geq1$. 

By using partition of unity, we extend the definitions above to the cases when $M$ is a smooth manifold and when 
we replace $M\times\bR^{N_1}$ by $T^*M$.
\end{defn}

Let $\Omega\subset M^3$ be an open coordinate patch. Let
$a(x, \xi)\in S^k_{1,0}(T^*\Omega)$. We define
\[A(x, y)=\frac{1}{(2\pi)^{3}}\int\! e^{i<x-y,\xi>}a(x,\xi)d\xi\]
as an oscillatory integral.  One can show that
\[A\colon C^\infty_0(\Omega)\To C^\infty(\Omega)\]
is continuous and has a unique continuous extension $A\colon\mathscr E'(\Omega)\to\mathscr D'(\Omega)$.

\begin{defn} \label{d:ss-pseudo}
Let $k\in\bR$. A classical pseudodifferential operator of order $k$ on $M$ is a continuous linear map
$A\colon C^\infty(M)\To\mathscr D'(M)$ such that on every open coordinate patch $\Omega$, if we consider $A$ as a continuous operator 
\[ A\colon C^\infty_0(\Omega)\To C^\infty(\Omega) , \]
then the distributional kernel of $A$ is
\[ A(x, y)=\frac{1}{(2\pi)^3}\int\! e^{i<x-y,\xi>}a(x, \xi)d\xi\]
with $a\in S^k_{{\rm cl\,}}(T^*\Omega)$. We call $a(x, \xi)$ the symbol of $A$. We write
$L^k_{{\rm cl\,}}(M)$
to denote the space of classical pseudodifferential operators of order $k$ on $M$.
\end{defn}

\section{The distributional kernel of $\tau$}\label{sec:hkernel}

In this section, we review some results in~\cite{Hsiao2014} about the orthogonal projection $\tau\colon L^2\to L^2\cap\mP$ which are needed in the proof of our main result. 

Let $\langle\,\cdot\,|\,\cdot\,\rangle$ be the Hermitian inner product on $T_{\bC}M$ given by 
\[ \langle Z_1 | Z_2 \rangle=-\frac{1}{2i}\langle\,d\theta\,,\,Z_1\wedge\overline{Z}_2\,\rangle \quad\text{for all $Z_1, Z_2\in T^{1,0}M$} . \]
The Hermitian metric $\langle\,\cdot\,|\,\cdot\,\rangle$ on $T_{\bC}M$ induces a Hermitian metric $\langle\,\cdot\,|\,\cdot\,\rangle$ on $T_{\bC}^*M$. Take $\theta\wedge d\theta$ to be the volume form on $M$, we then get natural inner product on $\Omega^{0,1}(M):=C^\infty(M,T^{*0,1}M)$ induced by $\theta\wedge d\theta$ and $\langle\,\cdot\,|\,\cdot\,\rangle$, where $T^{*0,1}M$ denotes the bundle of $(0,1)$ forms of $M$.  We denote this inner product by $(\,\cdot\,,\cdot\,)$ and denote the corresponding norm by $\norm{\cdot}$. Let $L^2_{(0,1)}(M)$ denote the completion of $\Omega^{0,1}(M)$ with respect to $(\,\cdot\,,\cdot\,)$. Let $\ddbar_b\colon C^\infty(M)\To\Omega^{0,1}(M)$ be the tangential Cauchy-Riemann operator. We extend $\ddbar_b$ to $L^2$ by $\ddbar_b\colon{\rm Dom\,}\ddbar_b\to L^2_{(0,1)}(M)$, where
\[ {\rm Dom\,}\ddbar_b:=\set{u\in L^2(M)\colon \, \ddbar_bu\in L^2_{(0,1)}(M)} . \]
Let $\ol{\pr}^*_b\colon {\rm Dom\,}\ol{\pr}^*_b\To L^2(M)$ be the $L^2$ adjoint of $\ddbar_b$. The Kohn Laplacian is given by
\begin{equation}\label{e-Box_b}
\begin{split}
&\Box_b:=\ddbar^{*}_b\ddbar_b\colon {\rm Dom\,}\Box_b\To L^2(M),\\
&{\rm Dom\,}\Box_b=\set{u\in L^2(M)\colon u\in{\rm Dom\,}\ddbar_b, \ddbar_bu\in{\rm Dom\,}\ol{\pr}^*_b}.
\end{split}
\end{equation}
Note that $\Box_b$ is self-adjoint.

The orthogonal projection $S\colon L^2(M)\to{\rm ker\,}\ddbar_b={\rm Ker\,}\Box_b$ is the \emph{Szeg\H{o} projection}. From now on, we assume that $M$ is embeddable. The follow facts are shown by the second-named author; see~\cite[Theorem 1.2 and Remark 1.4]{Hsiao2014}.

\begin{thm}\label{t-gue140812}
With the assumptions and notations above, we have 
\begin{equation}\label{e-gue140812a}
\tau=S+\ol S+F,
\end{equation}
where $F$ is a smoothing operator. Moreover, the kernel $\tau(x,y)\in\mathscr D'(M\times M)$ of $\tau$ satisfies 
\begin{equation}\label{e-gue140812aI}
\tau(x, y)\equiv\int^{\infty}_{0}e^{i\varphi(x, y)s}a(x, y, s)ds+\int^{\infty}_{0}e^{-i\ol\varphi(x, y)s}\ol a(x, y, s)ds,
\end{equation}
where
\begin{equation}\label{e-gue140812aIII}\begin{split}
&a(x, y, s)\in S^{1}_{{\rm cl\,}}\left(M\times M\times(0, \infty)\right), \\
&\mbox{$a(x, y, s)\sim\sum\limits^\infty_{j=0}a_j(x, y)s^{1-j}$ in $S^1_{1,0}\left(M\times M\times(0,\infty)\right)$},\\
&a_j(x,y)\in C^\infty(M\times M) \quad\text{for all $j\in\bN_0$},\\
&a_0(x,x)=\frac{1}{2}\pi^{-n} \quad\text{for all $x\in M$},
\end{split}\end{equation}
and
\begin{equation}\label{e-gue140812aII}
\begin{split}
&\varphi\in C^\infty(M\times M),\ \ {\rm Im\,}\varphi(x, y)\geq0,\ \ d_x\varphi|_{x=y}=-\theta(x),\\
&\varphi(x,y)=-\ol\varphi(y,x),\\
&\mbox{$\varphi(x,y)=0$ if and only if $x=y$},\\
&\mbox{$\sigma_{\Box_b}(x,\varphi'_x(x,y))$ vanishes to infinite order on $x=y$}.
\end{split}
\end{equation}
Here $\sigma_{\Box_b}$ denotes the principal symbol of $\Box_b$.
\end{thm}

We need the following fact about the Szeg\H{o} kernel (cf.\ \cite{BoutetSjostrand1976,Hsiao2010}).

\begin{thm}\label{t-gue140812I}
With the assumptions and notations above, the distributional kernel of $S$ satisfies 
\[S(x, y)\equiv\int^{\infty}_{0}e^{i\varphi(x, y)s}a(x, y,s)ds\]
where $\varphi(x,y)\in C^\infty(M\times M)$ and $a(x, y, s)\in S^{1}_{{\rm cl\,}}\left(M\times M\times(0,\infty)\right)$ are as in Theorem~\ref{t-gue140812}. 
\end{thm}

\section{The principal symbol of $\tau\Delta_b$ on $\mathcal{P}$}\label{sec:hsymbol}

It is well-known that $\Delta_b$ is a subelliptic operator.  However, if we restrict $\Delta_b$ to $\mP$, it is equivalent to an elliptic pseudodifferential operator.

\begin{thm}\label{t-gue140813}
There is a classical elliptic pseudodifferential operator $E_1\in L^1_{{\rm cl\,}}(M)$ with real-valued principal symbol such that 
\[\tau\Delta_b\tau=\tau E_1\tau\ \ \mbox{on $\mathscr D'(M)$}.\]
In particular, $\tau\Delta_b=\tau E_1$ on $\mathcal{P}$. 
\end{thm}

The proof of Theorem~\ref{t-gue140813} requires many ingredients.  First, we have the following immediate consequence of the commutator formulae proven by Lee~\cite{Lee1988}.

\begin{lem}\label{l-gue140813}
It holds that $\ol\Box_b=\Box_b+2iT+L$ for some $L\in C^\infty(M,T^{1,0}M\oplus T^{0,1}M)$. 
\end{lem}

We need the following result given in~\cite[Lemma 5.7]{HsiaoMarinescu2014}

\begin{lem}\label{l-gue140813I}
Let $A, B\colon C^\infty_0(M)\To\mathscr D'(M)$ be continuous operators such that the kernels of $A$ and $B$ satisfy 
\[\begin{split}
&A(x,y)=\int^\infty_0e^{i\varphi(x,y)s}\alpha(x,y,s)ds,\ \ \alpha(x,y,s)\in S^m_{{\rm cl\,}}(M\times M\times\ol\bR_+),\\
&B(x,y)=\int^\infty_0e^{-i\ol\varphi(x,y)s}\beta(x,y,s)ds,\ \ \beta(x,y,s)\in S^k_{{\rm cl\,}}(M\times M\times\ol\bR_+)
\end{split}\]
for some $m,k\in\bZ$, where $\varphi(x,y)\in C^\infty(M\times M)$ is as in Theorem~\ref{t-gue140812}. Then, 
\[A\circ B\equiv0,\ \ B\circ A\equiv0.\]
\end{lem}

To proceed, set 
\begin{equation}\label{e-gue140813aI}
\Sigma^-=\set{(x,\lambda\theta(x))\in T^\ast M;\, \lambda<0},\ \ \Sigma^+=\set{(x,\lambda\theta(x))\in T^*M;\, \lambda>0}.
\end{equation}
Let $\sigma_{\Box_b}(x,\xi)$ and $\sigma_{2iT}(x,\xi)$ be the principal symbols of $\Box_b$ and $2iT$, respectively. It is easy to see that $\sigma_{\Box_b}(x,\xi)=0$ for all $(x,\xi)\in\Sigma^-\cup\Sigma^+$; that $\sigma_{2iT}(x,\xi)>0$ for all $(x,\xi)\in\Sigma^-$; and that $\sigma_{2iT}(x,\xi)<0$ for all $(x,\xi)\in\Sigma^+$. For $(x,\xi)\in T^*M$, we write $\abs{\xi}$ to denote the point norm of the cotangent vector $\xi\in T^*_xM$. Take $\chi_0, \chi_1\in C^\infty(T^*M,[0,1])$ such that
\begin{enumerate}
\item $\chi_0=1$ in a small neighbourhood of $\Sigma^-\cap\set{(x,\xi)\in T^*M;\, \abs{\xi}\geq1}$,
\item $\chi_1=1$ in a small neighbourhood of $\Sigma^+\cap\set{(x,\xi)\in T^*M;\, \abs{\xi}\geq1}$,
\item $\supp\chi_0\cap\supp\chi_1=\emptyset$,
\item $\sigma_{2iT}(x,\xi)>0$ for all $(x,\xi)\in\supp\chi_0$,
\item $\sigma_{2iT}(x,\xi)<0$ for all $(x,\xi)\in\supp\chi_1$, and
\item $\chi_0$, $\chi_1$ are positively homogeneous of degree zero in the sense that 
\[\chi_0(x,\lambda\xi)=\chi_0(x,\xi),\ \ \chi_1(x,\lambda\xi)=\chi_1(x,\xi)\ \ \text{for all $\lambda\geq 1$ and $\abs{\xi}\geq1$}.\]
\end{enumerate}
Define
\begin{equation}\label{e-gue140813aII}
\begin{split}
q(x,\xi)=&(1-\chi_0(x,\xi)-\chi_1(x,\xi))\sqrt{\sigma_{\Box_b}(x,\xi)}\\
&\quad+\chi_0(x,\xi)\sigma_{2iT}(x,\xi)-\chi_1(x,\xi)\sigma_{2iT}(x,\xi).
\end{split}
\end{equation}
Note that $\sigma_{\Box_b}(x,\xi)>0$ for all $(x,\xi)\notin\Sigma^-\cup\Sigma^+$. From this observation, it is easy to see that $q(x,\xi)\geq c\abs{\xi}$ for all $(x,\xi)\in T^*M$ with $\abs{\xi}\geq1$, where $c>0$ is a constant. Let $\Td E_1\in L^1_{{\rm cl\,}}(M)$ with symbol $q(x,\xi)\in C^\infty(T^*M)$. Then $\Td E_1$ is a classical elliptic pseudodifferential operator. It is known that (see~\cite{Hsiao2010}) ${\rm WF'\,}(S)={\rm diag\,}(\Sigma^-\times\Sigma^-)$ and ${\rm WF'\,}(\ol S)={\rm diag\,}(\Sigma^+\times\Sigma^+)$, where 
\[{\rm WF'\,}(S)=\set{(x,\xi,y,\eta)\in T^*M\times T^*M\colon\, (x,\xi,y,-\eta)\in{\rm WF\,}(S)} \]
and ${\rm WF\,}(S)$ denotes the wave front set of $S$ in the sense of H\"ormander~\cite[Chapter 8]{Hormander2003}. 
Recall that $S$ denotes the Szeg\H{o} projection. From this observation and \eqref{e-gue140813aII}, it is not difficult to see that 
\begin{equation}\label{e-gue140813aIII}
S\Td E_1\equiv S(2iT),\ \ \Td E_1S\equiv(2iT)S,\ \ \ol S\Td E_1\equiv\ol S(-2iT),\ \ \Td E_1\ol S\equiv(-2iT)\ol S.
\end{equation}
Alternatively, \eqref{e-gue140813aIII} can be checked directly from the fact that $d_x\varphi|_{x=y}=-\theta(x)$. Now, we can prove the following theorem.

\begin{thm}\label{t-gue140814}
With the notations above, there is an $\Td E_0\in L^0_{{\rm cl\,}}(M)$ such that 
\[S\Delta_bS\equiv S(\Td E_1+\Td E_0)S \quad\text{and}\quad \ol S\Delta_b\ol S\equiv\ol S(\Td E_1+\Td E_0)\ol S.\]
\end{thm}

\begin{proof}
From Lemma~\ref{l-gue140813}, \eqref{e-gue140813aIII}, and the observation that $\Box_bS=0$, we have 
\begin{equation}\label{e-gue140814}
S\Delta_bS=S(2iT+L)S=S\Td E_1S+SLS+F_0,
\end{equation}
where $F_0\equiv0$. We write $L=U+\ol V$ for $U, V\in C^\infty(M,T^{1,0}M)$. Since $\ddbar_bS=0$, we have 
\begin{equation}\label{e-gue140814I}
S\ol VS=0.
\end{equation}
Now, 
\[(SUS)^*=SU^*S=S(-\ol U+r)S=SrS,\]
where $(SUS)^*$ and $U^*$ are the adjoints of $SUS$ and $S$ respectively and $r\in C^\infty(M)$. Hence, 
\begin{equation}\label{e-gue140814II}
SUS=S\ol rS.
\end{equation}
From \eqref{e-gue140814}, \eqref{e-gue140814I} and \eqref{e-gue140814II}, we conclude that 
\begin{equation}\label{e-gue140814III}
S\Delta_bS=S(\Td E_1+g_0)S+F_0,
\end{equation}
where $g_0\in C^\infty(M)$, $F_0\equiv0$. Similarly, 
\begin{equation}\label{e-gue140814IV}
\ol S\Delta_b\ol S=\ol S(\Td E_1+g_1)\ol S+F_1,
\end{equation}
where $g_1\in C^\infty(M)$ and $F_1\equiv0$. Put 
\[\Td E_0=\chi_0(x,\xi)g_0+\chi_1(x,\xi)g_1,\]
where $\chi_0, \chi_1$ are as in \eqref{e-gue140813aII}. As in the discussion before \eqref{e-gue140813aIII}, we have 
\begin{equation}\label{e-gue140814V}
Sg_0S\equiv S\Td E_0S,\ \ \ol Sg_1\ol S\equiv\ol S\Td E_0\ol S.
\end{equation}
The desired conclusion follows from~\eqref{e-gue140814III}, \eqref{e-gue140814IV} and~\eqref{e-gue140814V}.
\end{proof}

\begin{proof}[Proof of Theorem~\ref{t-gue140813}]
From Theorem~\ref{t-gue140814} and \eqref{e-gue140812a}, we have 
\begin{equation}\label{e-gue140814VI}
\begin{split}
\tau\Delta_b\tau&=(S+\ol S)\Delta_b(S+\ol S)+G_0\\
&=S\Delta_bS+\ol S\Delta_b\ol S+S\Delta_b\ol S+\ol S\Delta_bS+G_0\\
&=S(\Td E_1+\Td E_0)S+\ol S(\Td E_1+\Td E_0)\ol S+S\Delta_b\ol S+\ol S\Delta_bS+G_1\\
&=(S+\ol S)(\Td E_1+\Td E_0)(S+\ol S)-S(\Td E_1+\Td E_0)\ol S-\ol S(\Td E_1+\Td E_0)S\\
&\quad+S\Delta_b\ol S+\ol S\Delta_bS+G_1\\
&=\tau(\Td E_1+\Td E_0)\tau-S(\Td E_1+\Td E_0)\ol S-\ol S(\Td E_1+\Td E_0)S+S\Delta_b\ol S+\ol S\Delta_bS+G_2,
\end{split}
\end{equation}
where $G_0, G_1, G_2$ are smoothing operators. In view of Lemma~\ref{l-gue140813I} and Theorem~\ref{t-gue140812I}, we see that $S(\Td E_1+\Td E_0)\ol S$, $\ol S(\Td E_1+\Td E_0)S$, $S\Delta_b\ol S$ and
$\ol S\Delta_bS$ are smoothing. From this and \eqref{e-gue140814VI}, we get 
\[\tau\Delta_b\tau=\tau(\Td E_1+\Td E_0)\tau+G,\]
where $G$ is smoothing. Hence,
\begin{equation}\label{e-gue140814VII}
\tau\Delta_b\tau=\tau^2\Delta_b\tau^2=\tau^2(\Td E_1+\Td E_0)\tau^2+\tau G\tau=\tau(\Td E_1+\Td E_0+G)\tau.
\end{equation}
Put $E_1=\Td E_1+\Td E_0+G\in L^1_{{\rm cl\,}}(M)$. From \eqref{e-gue140814VII}, we get $\tau\Delta_b\tau=\tau E_1\tau$. The theorem follows.
\end{proof}

\section{The Green's function of square root of $\oP^\prime_4$}\label{sec:hsqrt}

In this section, we will prove Theorem~\ref{t-gue140827}. First, we can repeat the proof of Theorem~\ref{t-gue140813} with minor change and get the following result.

\begin{thm}\label{t-gue140814I}
We have 
\[ \oP_4^\prime = \tau\bigr((2E_1)^2+\hat E_1\bigr)\ \ \mbox{on $\mathcal{P}$},\]
where $E_1\in L^1_{{\rm cl\,}}(M)$ is as in Theorem~\ref{t-gue140813} and $\hat E_1\in L^1_{{\rm cl\,}}(M)$.
\end{thm}

In particular, $\oP_4^\prime$ is an elliptic pseudodifferential operator  
on $\mathcal{P}$. Standard arguments for elliptic operators imply that the spectrum ${\rm Spec\,}\oP_4^\prime$ of $\oP_4^\prime$ is a discrete subset of $\left(-\infty,\infty\right)$ such that every $\lambda\in{\rm Spec\,}\oP_4^\prime$ is an eigenvalue of $\oP_4^\prime$ and the eigenspace 
\[\mathcal{E}_\lambda(\oP_4^\prime):=\set{u\in{\rm Dom\,}\oP_4^\prime \colon \oP_4^\prime u=\lambda u} \]
is a finite dimensional subspace of $\mathcal{P}$.

Let 
\[\pi\colon\hat{\mathcal{P}}\to{\rm Ker\,}\oP_4^\prime\]
be the orthogonal projection. Let $\set{g_1,g_2,\dotsc,g_d}\subset\mP$ be an orthonormal frame for ${\rm Ker\,}\oP_4^\prime$, where $d\in\mathbb N_0$. Then
\begin{equation}\label{e-gue140815}
\pi(x,y)=\sum^d_{j=1}g_j(x)\ol g_j(y)\in C^\infty(M\times M).
\end{equation}
From \eqref{e-gue140815}, we can extend $\pi$ to $\mathscr D'(M)$ as a smoothing operator on $M$. 

Assume that $\oP_4^\prime$ is nonnegative. Then ${\rm Spec\,}\oP_4^\prime\subset\left[0,\infty\right)$ and $\oP_4^\prime$ has a well-defined square root 
\[ \bigl(\oP_4^\prime\bigr)^{\frac{1}{2}} \colon {\rm Dom\,}\bigl(\oP_4^\prime\bigr)^{\frac{1}{2}} \subset \hat\mP \to \hat\mP . \]
Note that ${\rm Dom\,}\bigl(\oP_4^\prime\bigr)^{\frac{1}{2}}={\rm Dom\,}\oP_4^\prime$. We write 
\[\bigl(\oP_4^\prime\bigr)^{-\frac{1}{2}}\colon\hat\mP\To{\rm Dom\,}\bigl(\oP_4^\prime\bigr)^{\frac{1}{2}}\]
to denote the Green's function of $\bigl(\oP_4^\prime\bigr)^{\frac{1}{2}}$. That is, 
\begin{equation}\label{e-gue141204}
\begin{split}
&\bigl(\oP_4^\prime\bigr)^{\frac{1}{2}}\circ\bigl(\oP_4^\prime\bigr)^{-\frac{1}{2}}+\pi=I\ \ \mbox{on $\hat\mP$},\\
&\bigl(\oP_4^\prime\bigr)^{-\frac{1}{2}}\circ\bigl(\oP_4^\prime\bigr)^{\frac{1}{2}}+\pi=I\ \ \mbox{on ${\rm Dom\,}\bigl(\oP_4^\prime\bigr)^{\frac{1}{2}}$}.
\end{split}
\end{equation}

For every $t>0$, the operator 
\[ \oP_4^\prime+t+\pi \colon {\rm Dom\,}\oP_4^\prime\to\hat{\mathcal{P}}\]
has a continuous inverse 
\[(\oP_4^\prime+t+\pi)^{-1}\colon\hat{\mathcal{P}}\to\hat{\mathcal{P}}\]
and the operator $(\oP_4^\prime+t+\pi)^{-1}$ depends continuously on $t$. Let $\lambda_1>0$ be the first non-zero eigenvalue of $\oP_4^\prime$. Then
\begin{equation}\label{e-gue140816f}
\begin{split}
\norm{(\oP_4^\prime+t+\pi)^{-1}u} & \leq \frac{1}{\lambda_1+t}\norm{(I-\pi)u}+\frac{1}{1+t}\norm{\pi u} \\
& \leq \frac{1}{\min\{\lambda_1,1\}+t}\norm{u}
\end{split}
\end{equation}
for all $u\in\hat\mP$.  $\bigl(\oP_4^\prime\bigr)^{-1/2}$ can be understood as follows.

\begin{lem}\label{l-gue140815}
On $\hat{\mathcal{P}}\cap({\rm Ker\,}\oP^\prime_4)^\perp$, we have
\[ \bigl(\oP_4^\prime\bigr)^{-\frac{1}{2}} = c\int^\infty_0t^{-\frac{1}{2}}(\oP_4^\prime+t+\pi)^{-1}dt ,\]
where $c^{-1}=\int^\infty_0t^{-\frac{1}{2}}(1+t)^{-1}dt$.
\end{lem}

\begin{proof}
Fix a positive eigenvalue $\lambda\in{\rm Spec\,}\oP_4^\prime$. Let $u\in\mathcal{E}_\lambda(\oP_4^\prime)$. Then, 
\begin{equation}\label{e-gue140815I}
\bigl(\oP_4^\prime\bigr)^{-\frac{1}{2}}u=\frac{1}{\sqrt{\lambda}}u.
\end{equation}
We compute that 
\begin{equation}\label{e-gue140815II}
\Bigr(c\int^\infty_0t^{-\frac{1}{2}}({\rm P'_0\,}+t+\pi)^{-1}dt\Bigr)u = cu\int^\infty_0t^{-\frac{1}{2}}\frac{1}{\lambda+t}dt = \frac{1}{\sqrt{\lambda}}u.
\end{equation}
Hence the conclusion is true on $\mE_\lambda(\oP_4^\prime)$ for all $\lambda\in{\rm Spec\,}\oP_4^\prime$.

Let $u\in\hat{\mathcal{P}}\cap({\rm Ker\,}\oP^\prime_4)^\perp$. For each $N\in\bN$, let $u_N$ be the orthogonal projection of $u$ onto $\bigoplus_{\lambda\leq N}\mE_\lambda(\oP_4^\prime)$.  It follows that $u_N\to u$ and that $\bigl(\oP_4^\prime\bigr)^{-\frac{1}{2}}u_N\to\bigl(\oP_4^\prime\bigr)^{-\frac{1}{2}}u$. From \eqref{e-gue140816f}, we have 
\[c\Bigl(\int^\infty_0t^{-\frac{1}{2}}(\oP^\prime_4+t+\pi)^{-1}dt\Bigr)u_N\To c\Bigl(\int^\infty_0t^{-\frac{1}{2}}(P^\prime_4+t+\pi)^{-1}dt\Bigr)u\]
in $\hat\mP$ as $N\to\infty$.  Together these observations yield the result.
\end{proof}

To proceed, we require some additional symbol spaces.

\begin{defn}\label{d-gue140815}
Let $m$ be real number. The class $S^m_{1,0,d}(T^*M,\bR_+)$ consists of all functions $a(x,\xi,t)\in C^\infty(T^*M\times\bR_+)$ such that for arbitrary multi-indices $\alpha, \beta\in\mathbb N^3_0$, and for any compact set $K\subset M$ there exists $C_{\alpha,\beta,K}>0$ such that
$\abs{\pr^\alpha_x\pr^\beta_\xi a(x,\xi,t)}\leq C_{\alpha,\beta,K}(1+\abs{\xi}+\abs{t}^{\frac{1}{d}})^{m-\abs{\beta}}$ for all $(x,\xi)\in T^*K$, $t\in\bR_+$. Denote 
\[S^{-\infty}(T^*M,\bR_+)=\bigcap_{m\in\bR}S^m_{1,0,d}(T^*M,\bR_+).\]

Let $a_j\in S^{m_j}_{1,0,d}(T^*M,\bR_+)$ for $j\in\{0,1,2,\dotsc\}$ with $m_j\to-\infty$ as $j\to\infty$. Then there exists $a\in S^{m_0}_{1,0,d}(T^*M,\bR_+)$, unique modulo $S^{-\infty}(T^*M,\bR_+)$, such that 
$a-\sum\limits^{k-1}_{j=1}a_j\in S^{m_k}_{1,0,d}(T^*M,\bR_+)$ for $k\in\{0,1,2,\dotsc\}$. If $a$ and $a_j$ have the properties above, we write 
\[\mbox{$a\sim\sum\limits^\infty_{j=0}a_j$ in $S^{m_0}_{1,0,d}(T^*M,\bR_+)$}.\]

Let $S^m_{{\rm cl\,},d}(T^*M,\bR_+)$ be the space of all symbols $a(x,\xi,t)\in S^m_{1,0,d}(T^*M,\bR_+)$ with $a(x,\xi,t)\sim\sum\limits^\infty_{j=0}a_{m-j}(x,\xi,t)$ in $S^m_{1,0,d}(T^*M,\bR_+)$, where $a_{m-j}(x,\xi,t)$ is positively homogeneous  of degree $m-j$ in $(\xi,t^{\frac{1}{d}})$; i.e. 
\[a_{m-j}(x,\lambda\xi,\lambda^dt)=\lambda^{m-j}a_{m-j}(x,\xi,t),\ \ \mbox{for $t\in\bR_+$, $\lambda\geq1$, $\abs{\xi}\geq1$}.\]
\end{defn}

Let $a(x,\xi,t)\in S^m_{{\rm cl\,},d}(T^*M,\bR_+)$.  We construct a pseudodifferential operator $P_t$, depending smoothly on $t$, by
\[(P_tu)(x)=\frac{1}{(2\pi)^3}\int e^{i<x-y,\xi>}a(x,\xi,t)u(y)dyd\xi \quad \text{for all $u\in C^\infty(M)$} . \]
We call $a(x,\xi,t)$ the symbol of $P_t$ and $a_m(x,\xi,t)$ the principal symbol of $P_t$. In this case, we will write $P_t\in L^m_{{\rm cl\,},d}(M,\bR_+)$. 

Let $P_t\in L^{-2}_{{\rm cl\,},2}(M,\bR_+)$. Then $P_t\colon H^s(M)\To H^{s+2}(M)$ is continuous for all $s\in\mathbb Z$ and all $t\in\bR_+$. Let $f(t)$ be a strictly positive continuous function. We write
\[P_t=O(f(t))\colon H^{s_1}(M)\To H^{s_2}(M),\ \ s_1, s_2\in\mathbb Z,\]
if $\norm{P_tu}_{s_2}\leq Cf(t)\norm{u}_{s_1}$ for all $u\in H^{s_1}(M)$ and all $t\in\bR_+$, where $\norm{\cdot}_s$ denotes the standard Sobolev norm of order $s$ and $C>0$ is a constant independent of $t$.

We return to our situation. Put 
\begin{equation}\label{e-gue141204a}
E_2=(2E_1)^2+\hat E_1,
\end{equation} 
where $E_1,\hat E_1\in L^1_{{\rm cl\,}}(M)$ are as in Theorem~\ref{t-gue140814I}. Let $e_2(x,\xi)\in S^2_{{\rm cl\,}}(T^*M)$ be the principal symbol of $E_2$. The following is well-known~\cite[Chapter 2]{Shubin2001}.

\begin{thm}\label{t-gue140816}
There exists $G_t\in L^{-2}_{{\rm cl\,},2}(M,\bR_+)$ depending continuously on $t$ in $L^2(M)$ such that 
\begin{align}
&G_t=O(\frac{1}{1+t})\colon H^s(M)\to H^{s}(M)&&\text{for all $s\in\bZ$},\label{e-gue141204aI}\\
&G_t=O(\frac{1}{\sqrt{1+t}})\colon H^s(M)\to H^{s+1}(M)&&\text{for all $s\in\bZ$},\label{e-gue141204aII}\\
&G_t=O(1)\colon H^s(M)\to H^{s+2}(M)&&\text{for all $s\in\bZ$},\label{e-gue140816}\\
&g_0(x,\xi,t)=\frac{1}{e_2(x,\xi)+t}&&\text{for all $\abs{\xi}\geq1$},\label{e-gue140816I}\\
&(E_2+t)G_t=I+F_t&&\text{for all $t>0$},\label{e-gue140816II}
\end{align}
where $g_0(x,\xi,t)$ denotes the principal symbol of $G_t$ and $F_t$ is a smoothing operator on $M$ depending smoothly on $t$ with the property that for all $m\in\bN_0$, there is a constant $C_m>0$ such that for all $t\in\bR_+$, 
\begin{equation}\label{e-gue140816III}
\abs{F_t(x,y)}_{C^m(M\times M)}\leq C_m\frac{1}{1+t}.
\end{equation}

Moreover, in local coordinates $x$, let $g(x,\xi,t)$ denote the full symbol of $G_t$. Then, for every $\alpha, \beta\in\mathbb N^3_0$, there is a constant $C_{\alpha,\beta}>0$, independent of $t$, such that 
\begin{align}
\label{e-gue141227} \abs{\pr^\alpha_x\pr^\beta_\xi g(x,\xi,t)} & \leq C_{\alpha,\beta}\frac{1}{\sqrt{1+t}}(1+\abs{\xi})^{-1-\abs{\beta}} \quad\text{for all $\abs{\xi}\geq1$}, \\
\label{e-gue141227I} \abs{\pr^\alpha_x\pr^\beta_\xi g(x,\xi,t)} & \leq C_{\alpha,\beta}\frac{1}{1+t}(1+\abs{\xi})^{-\abs{\beta}} \quad\text{for all $\abs{\xi}\geq1$} .
\end{align}
\end{thm}

We introduce some notations. Let $\vartheta(x,y)$ denote the Carnot--Carath\'eodory distance on $(M^3,T^{1,0}M,\theta)$.  Let $(z,t)$ be CR normal coordinates defined in a neighborhood of $p\in M$ such that $(z(p),t(p))=(0,0)$. Define $\rho^4(z,t)=\abs{z}^4+t^2$. It is easy to see (cf.\ \cite[Section~3]{HsiaoYung2014}) that for points $x$ sufficiently close to $p$, we have
\[ \vartheta(x,p) \simeq \rho(x). \]
Denote by $B(x,r)$ the non-isotropic ball $\{y \in M \colon \vartheta(x,y) < r\}$ of radius $r$ centered at $x$. 
Let $k\in\mathbb N$. We denote by $\nabla^k_b$ any differential operator of the form $L_1\ldots L_k$, where $L_j\in C^\infty(M,T^{1,0}M\oplus T^{0,1}M)$ satisfy $\langle\,L_j\,|\,L_j\,\rangle\leq1$ for $j=1,\dotsc,k$.

Next, we define a class of (non-isotropic) smoothing operators of order $j$. For our purposes, it suffices to restrict to the case when $0 \leq j < 4$.

Recall that a smooth function $\phi$ on $M$ is said to be a \emph{normalized bump function} on $B(x,r)$ if $\supp\phi\subset B(x,r)$ and
\begin{equation}
\label{eqn:normalized_bump}
\norm{\nabla_b^k \phi}_{L^{\infty}(B(x,r))} \leq C_k r^{-k}
\end{equation} 
for all $k \geq 0$; here $C_k>0$ are absolute constants independent of $r$.  If~\eqref{eqn:normalized_bump} only holds for $0\leq k\leq N$ for some large integer $N$, we say that $\phi$ is a normalized bump function of order $N$ in $B(x,r)$.

Suppose that $A$ is a continuous linear operator $A \colon C^{\infty}(M) \to C^{\infty}(M)$ and its adjoint $A^*$ is also a continuous map  $A^* \colon C^{\infty}(M) \to C^{\infty}(M)$. We say that $A$ is a smoothing operator of order $j$,  $0 \leq j < 4$, if 
\begin{enumerate}
 \item there exists a kernel $A(x,y)$, defined and smooth away from the diagonal in $M\times M$, such that 
 \begin{equation} \label{eq:Tkernelrep}
 Af(x) = \int_{M} A(x,y) f(y)dv_M(y)
 \end{equation}
 for any $f \in C^{\infty}(M)$, and every $x\not\in\supp f$, where $dv_M=\theta\wedge d\theta$;
 \item for all $x\not=y$, the kernel $A(x,y)$ satisfies
 \[ |(\nabla_b)_x^{\alpha_1} (\nabla_b)_y^{\alpha_2} A(x,y)| \lesssim_{\alpha} \vartheta(x,y)^{-4+j-|\alpha|} \quad \text{for all $|\alpha| = |\alpha_1| + |\alpha_2|$}; \]
 \item the operators $A$ and $A^*$ satisfy the following cancellation conditions of order $j$: if $\phi$ is a normalized bump function in $B(x,r)$, then
 \begin{align*}
  \|\nabla_b^{\alpha} A\phi \|_{L^{\infty}(B(x,r))} & \lesssim_{\alpha} r^{j-|\alpha|}, \\
  \|\nabla_b^{\alpha} A^*\phi \|_{L^{\infty}(B(x,r))} & \lesssim_{\alpha} r^{j-|\alpha|}.
 \end{align*}
\end{enumerate}

Since $M$ is embeddable, $\Box_b$ has $L^2$ closed range. Let 
\begin{equation}\label{e-gue141204b}
N\colon L^2(M)\To{\rm Dom\,}\Box_b
\end{equation}
be the partial inverse of $\Box_b$ and let $N(x,y)$ be the distributional kernel of $N$. The following is well-known (see~\cite[Theorem~2.2]{HsiaoYung2014})

\begin{thm} \label{thm2.1}
The Szeg\H{o} projection $S$ and the partial inverse $N$ of $\Box_b$ are smoothing operators of orders $0$ and $2$, respectively. 
\end{thm}

We also need to study one-parameter families of smooth operators.

\begin{defn}\label{d-gue141226}
Let $A_t$ be a $t$-dependent smoothing operator of order $j$, $0\leq j<4$, where $t\in\mathbb R_+$. Let $f(t)$ be a positive continuous function of $t\in\mathbb R_+$. We say that $A_t$ is a smoothing operator of order $j$ with size $f(t)$ if for every $m\in\mathbb N_0$ and any normalized bump function $\phi$ in $B(x,r)$, there are constants $C_m,C_{m,r}>0$, independent of $t$, such that for all $t\in\mathbb R_+$,
\begin{align*}
 \lvert(\nabla_b)_x^{\alpha_1} (\nabla_b)_y^{\alpha_2}A_t(x,y)\rvert & \leq C_{m}f(t)\vartheta(x,y)^{-4+j-|\alpha|} \quad \text{for all $|\alpha| = |\alpha_1| + |\alpha_2|\leq m$} , \\
 \lVert\nabla_b^{\alpha}A_t\phi \rVert_{L^{\infty}(B(x,r))} & \leq f(t)C_{m,r}r^{j-|\alpha|} \quad \text{for all $\abs{\alpha}\leq m$} , \\
 \lVert\nabla_b^{\alpha}A_t^*\phi \rVert_{L^{\infty}(B(x,r))} & \leq f(t)C_{m,r} r^{j-|\alpha|} \quad \text{for all $\abs{\alpha}\leq m$} .
\end{align*}
\end{defn}

We also need the following result~\cite[Theorem~2.2 and Theorem~2.3]{HsiaoYung2014}.

\begin{thm} \label{thm2.2}
Let $A_t$ and $B_t$ be $t$-dependent smoothing operators of orders $j_1$ and $j_2$ with sizes $f(t)$ and $g(t)$, respectively, where $j_1, j_2 \geq 0$, $j_1 + j_2 < 4$, and $f(t), g(t)$ are positive continuous functions. Then $A_t\circ B_t$ is a smoothing operator of order $j_1 + j_2$ with size $f(t)g(t)$.
\end{thm}

Let $P_t\in L^{-2}_{{\rm cl\,},2}(M,\bR_+)$, $Q_t\in L^{-1}_{{\rm cl\,},2}(M,\bR_+)$, and $R_t\in L^0_{{\rm cl\,},2}(M,\bR_+)$. 
Let $p(x,\xi,t)$, $q(x,\xi,t)$ and $r(x,\xi,t)$ be symbols of $P_t$, $Q_t$ and $R_t$ respectively. It is easy to see that for every $\alpha, \beta\in\mathbb N^3_0$, there is a constant $C_{\alpha,\beta}>0$ independent of $t$ such that 
\begin{equation}\label{e-gue141227II}
\begin{split}
&\abs{\pr^\alpha_x\pr^\beta_\xi p(x,\xi,t)}\leq C_{\alpha,\beta}\frac{1}{1+t}(1+\abs{\xi})^{-\abs{\beta}}\quad \text{for all $\abs{\xi}\geq1$},\\
&\abs{\pr^\alpha_x\pr^\beta_\xi q(x,\xi,t)}\leq C_{\alpha,\beta}\frac{1}{\sqrt{1+t}}(1+\abs{\xi})^{-\abs{\beta}} \quad \text{for all $\abs{\xi}\geq1$} ,\\
&\abs{\pr^\alpha_x\pr^\beta_\xi r(x,\xi,t)}\leq C_{\alpha,\beta}(1+\abs{\xi})^{-\abs{\beta}}\quad \text{for all $\abs{\xi}\geq1$}.
\end{split}
\end{equation}

Using the following lemma, we establish an analogue of Theorem~\ref{thm2.1}.

\begin{lem}\label{l-gue150626}
Consider $B(x,r)$, where $x\in M$ and $r>0$ is a small constant. Let $\chi_r\in C^\infty_0((B(x,2r))$ be a normalized bump function on $B(x,2r)$ with $\chi_r\equiv1$ on $B(x,r)$. There is a constant $C>0$ independent of $r$ such that 
\begin{equation}\label{e-gue150626}
\norm{f}_{L^\infty(B(x,r))}\leq Cr\sum^3_{j=0}\norm{\nabla^j_b(\chi_rf)} \quad \text{for all $f\in C^\infty(M)$}.
\end{equation}
\end{lem}

\begin{proof}
Consider $\Delta_b+I\colon {\rm Dom\,}(\Delta_b+I)\subset L^2(M)\To L^2(M)$, where ${\rm Dom\,}(\Delta_b+I)=\set{u\in L^2(M)\colon (\Delta_b+I)u\in L^2(M)}$. It is clear that $\Delta_b+I$ is injective, self-adjoint, has $L^2$ closed range and hence is surjective. Let $H\colon L^2(M)\To{\rm Dom\,}(\Delta_b+I)$ be the inverse of $\Delta_b+I$. Put $B:=H^2\colon L^2(M)\To L^2(M)$. We have 
\begin{equation}\label{e-gue150626I}
B(\Delta_b+I)^2=I\ \ \mbox{on $C^\infty(M)$}.
\end{equation}
It is known that (see~\cite[Appendix~A]{HsiaoYung2014}) 
\begin{equation}\label{e-gue150626a}
\begin{split}
&\mbox{$B$ is a smoothing operator of order $4-\varepsilon$ for every $\varepsilon>0$},\\
&\mbox{$B\nabla_b$ is a smoothing operator of order $3$}.
\end{split}
\end{equation}
Let $f\in C^\infty(M)$. From \eqref{e-gue150626I}, we have 
\begin{equation}\label{e-gue150626II}
\chi_rf=B(\Delta_b+I)^2\chi_rf=\sum^4_{j=0}B\nabla^j_b\chi_rf.
\end{equation}
Fix $x_0\in B(x,r)$. From \eqref{e-gue150626II}, we have 
\begin{align*}
f(x_0)&=(\chi_rf)(x_0)=(B\nabla^4_b\chi_rf)(x_0)+\sum^3_{j=0}(B\nabla^j_b\chi_rf)(x_0)\\
&=\int(B\nabla_b)(x_0,y)\nabla^3_b(\chi_rf)(y)dv_M(y)+\sum^3_{j=0}\int B(x_0,y)\nabla^j_b(\chi_rf)(y)dv_M(y),
\end{align*}
where $(B\nabla_b)(x,y)$ and $B(x,y)$ denote the distribution kernels of $B\nabla_b$ and $B$ respectively. We then check that 
\begin{equation}\label{e-gue150626IV}
\begin{split}
\abs{f(x_0)}\leq&\Bigr(\int_{B(x_0,2r)}\abs{(B\nabla_b)(x_0,y)}^2dv_M(y)\Bigr)^{\frac{1}{2}}\norm{\nabla^3_b(\chi_rf)}\\
&\quad\quad+\sum^3_{j=0}\Bigr(\int_{B(x_0,2r)}\abs{B(x_0,y)}^2dv_M(y)\Bigr)^{\frac{1}{2}}\norm{\nabla^j_b(\chi_rf)}.
\end{split}
\end{equation}
From \eqref{e-gue150626a}, we can check that 
\begin{equation}\label{e-gue150626V}
\begin{split}
&\int_{B(x_0,2r)}\abs{(B\nabla_b)(x_0,y)}^2dv_M(y)\leq C_0\int_{B(x_0,2r)}\vartheta(x_0,y)^{-2}dy\leq C_1r^2,\\
&\int_{B(x_0,2r)}\abs{B(x_0,y)}^2dv_M(y)\leq C_2\int_{B(x_0,2r)}\vartheta(x_0,y)^{-2}dy\leq C_3r^2,
\end{split}
\end{equation}
where $C_0, C_1, C_2, C_3$ are positive constants independent of $r$ and the point $x_0$. From \eqref{e-gue150626V} and \eqref{e-gue150626IV}, \eqref{e-gue150626} follows. 
\end{proof}

\begin{thm}\label{t-gue141227}
The operators $SP_t$, $SQ_t$ and $SR_t$ are smoothing operators of orders $0$ with sizes $\frac{1}{1+t}$, $\frac{1}{\sqrt{1+t}}$ and $1$, respectively. Similarly, the operators $P_tS$, $Q_tS$ and $R_tS$ are smoothing operators of orders $0$ with sizes $\frac{1}{1+t}$, $\frac{1}{\sqrt{1+t}}$ and $1$, respectively.
\end{thm}

\begin{proof}
Let $\phi$ be a normalized bump function in the ball $B(x,r)$. From  \eqref{e-gue150626}, we have 
\begin{equation}\label{e-gemimapaVbos}
\norm{SQ_t\phi}_{L^\infty(B(x,r))}\leq Cr\sum^3_{j=0}\norm{\nabla^j_b(\chi_rSQ_t\phi)},
\end{equation}
where $\chi_r$ is as in Lemma~\ref{l-gue150626} and $C>0$ is a constant independent of $r$, $\phi$, $x$ and $t$. We claim that 
\begin{equation}\label{e-gue150104}
\norm{\nabla^{j}_bSQ_t\phi}\leq c_j\frac{1}{\sqrt{1+t}}r^{2-j} \text{for $j=0,1,2,3,4$},
\end{equation}
where $c_j>0$ is a constant independent of $r$, $x$ and $t$. Fix $j\in\set{0,1,2,3,4}$. It is known that (see~\cite{Hsiao2010,SteinYung13})
\begin{equation}\label{e-gue150104II}
\nabla^{2j}_bS\colon H^s(M)\To H^{s-j}(M) \quad\text{for all $s\in\bZ$} .
\end{equation}
Moreover, from \eqref{e-gue141227II}, we can check that 
\begin{equation}\label{e-gue150105}
Q_t=O(\frac{1}{\sqrt{1+t}})\colon H^s(M)\To H^s(M) \quad\text{for all $s\in\bZ$}.
\end{equation}
From \eqref{e-gue150104II} and \eqref{e-gue150105}, we deduce that
\begin{equation}\label{e-gue150529}
\nabla^{2j}_bSQ_t=O(\frac{1}{\sqrt{1+t}})\colon H^s(M)\To H^{s-j}(M) \quad\text{for all $s\in\bZ$} .
\end{equation}
From \eqref{e-gue150529}, we have 
\begin{equation}\label{e-gue150528I}
\begin{split}
\norm{\nabla^{j}_bSQ_t\phi}^2&=(\,\nabla^{j}_bSQ_t\phi\,|\,\nabla^{j}_bSQ_t\phi\,)=(\,\nabla^{2j}_bSQ_t\phi\,|\,SQ_t\phi\,)\\
&\lesssim\norm{\nabla^{2j}_bSQ_t\phi}\norm{SQ_t\phi}\\
&\lesssim\frac{1}{1+t}\norm{\phi}_j\norm{\phi}\\
&\lesssim\frac{1}{1+t}\norm{\nabla^{2j}_b\phi}\norm{\phi}\\
&\lesssim\frac{1}{1+t}r^{4-2j},
\end{split}
\end{equation}
where $\norm{\phi}_{j}$ denotes the standard Sobolev norm of $\phi$ of order $j$. 
From \eqref{e-gue150528I}, the claim \eqref{e-gue150104} follows. 

From~\eqref{e-gemimapaVbos} and~\eqref{e-gue150104} we can check that 
\[\begin{split}
\norm{SQ_t\phi}_{L^\infty(B(x,r))}&\lesssim r\sum^3_{j=0}\norm{\nabla^{j}_b(\chi_rSQ_t\phi)}\\
&\lesssim r\norm{\chi_rSQ_t\phi}+r\sum^3_{j=1}\sum^j_{s=0}r^{-j+s}\norm{\nabla^{s}_b(SQ_t\phi)}\\
&\lesssim r\frac{1}{\sqrt{1+t}}+r\sum^3_{j=1}\sum^j_{s=0}r^{-j+s}\frac{1}{\sqrt{1+t}}r^{2-s}\\
&\lesssim\frac{1}{\sqrt{1+t}}.\end{split}\]
We can repeat the method above with minor changes and get that for every $j\in\mathbb N_0$, $\norm{\nabla^j_bSQ_t\phi}\lesssim_{j}\frac{1}{\sqrt{1+t}}r^{-j}$. Thus, $SQ_t$ satisfies the cancellation condition of order $0$ with size $\frac{1}{\sqrt{1+t}}$. Similarly, we can repeat the procedure above with minor change and obtain that $(SQ_t)^*$ satisfies the cancellation condition of order $0$ with size $\frac{1}{\sqrt{1+t}}$.

Now, we estimate the kernel $SQ_t(x,y)$. Let $x=(x_1,x_2,x_3)$ be local coordinates for $M$ defined in an open set $D\subset M$. From Theorem~\ref{t-gue140812I} and the complex stationary phase formula of Melin--Sj\"ostrand~\cite{MS74}, it follows that 
\begin{equation}\label{e-gue150105I}
(SQ_t)(x,y)=\int^\infty_0e^{i\varphi(x,y)s}b(x,y,s,t)ds+F_t(x,y)\ \ \mbox{on $D\times D$},
\end{equation}
where $b(x,y,s,t)\in C^\infty(D\times D\times\mathbb R_+\times\mathbb R_+)$, and for every $\alpha, \beta\in\mathbb N^3_0$, $\gamma\in\mathbb N_0$, there is a constant $C_{\alpha,\beta,\gamma}>0$, independent of $t$, such that on $D\times D$,
\begin{equation}\label{e-gue150105II}
\begin{cases}
\abs{\pr^\alpha_x\pr^\beta_y\pr^\gamma_sb(x,y,s,t)}\leq  C_{\alpha,\beta,\gamma}(\sqrt{s^2+t})^{-\gamma},& \mbox{if $\gamma\geq1$}\\
\abs{\pr^\alpha_x\pr^\beta_yb(x,y,s,t)}\leq  C_{\alpha,\beta,\gamma}(\frac{s}{\sqrt{s^2+t}}),& \mbox{if $\gamma=0$},
\end{cases}
\end{equation}
and $F_t$ is a smoothing operator on $D$ depending smoothly on $t$ with the property that for all $m\in\bN_0$, there is a constant $C_m>0$ such that for all $t\in\bR_+$, 
\[\abs{F_t(x,y)}_{C^m(M\times M)}\leq C_m\frac{1}{\sqrt{1+t}}.\]

From~\eqref{e-gue150105II}, the formula
\[\int^\infty_0e^{i\varphi(x,y)s}b(x,y,s,t)ds=\int^\infty_0\frac{1}{(i\varphi(x,y))^2}\frac{\pr^2}{\pr s^2}(e^{i\varphi(x,y)s})b(x,y,s,t)ds,\]
and distribution theory, one can check that 
\begin{multline}
 \label{e-gue150105III}
 \int^\infty_0e^{i\varphi(x,y)s}b(x,y,s,t)ds\\
 =\int^\infty_0\frac{1}{(i\varphi(x,y))^2}e^{i\varphi(x,y)s}\frac{\pr^2}{\pr s^2}b(x,y,s,t)ds+\frac{1}{(i\varphi(x,y))^2}H_t(x,y),
\end{multline}
where $H_t$ is a smoothing operator on $D$ depending smoothly on $t$ with the property that for all $m\in\bN_0$, there is a constant $\Td C_m>0$ such that for all $t\in\bR_+$, 
\[\abs{H_t(x,y)}_{C^m(M\times M)}\leq\Td C_m\frac{1}{\sqrt{1+t}}.\]

Again, from \eqref{e-gue150105II}, we have 
\begin{equation}\label{e-gue150106}
\begin{split}
\abs{\int^\infty_0\frac{1}{(i\varphi(x,y))^2}e^{i\varphi(x,y)s}\frac{\pr^2}{\pr s^2}b(x,y,s,t)ds}&\leq\hat C\frac{1}{\abs{\varphi(x,y)}^2}\int^\infty_0\frac{1}{s^2+t}ds\\
&\leq\hat C_1\frac{1}{\sqrt{1+t}}\frac{1}{\abs{\varphi(x,y)}^2},
\end{split}
\end{equation}
for all $t\geq1$, where $\hat C>0$, $\hat C_1>0$ are constants independent of $t$. It is known that (see~\cite[Theorem~1.4]{Hsiao2010}) 
$\abs{\varphi(x,y)}\approx\vartheta(x,y)^2$. From this observation, \eqref{e-gue150105III} and \eqref{e-gue150106}, we conclude that 
\[\abs{(SQ_t)(x,y)}\leq C\frac{1}{\sqrt{1+t}}\vartheta(x,y)^{-4} \]
for all $x,y\in M$ with $x\not=y$, where $C>0$ is a constant independent of $t$. 

For every $m\in\mathbb N$, we can repeat the procedure above with minor change  and deduce that there is a constant $C_m>0$ independent of $t$ such that 
\[ |(\nabla_b)_x^{\alpha_1} (\nabla_b)_y^{\alpha_2}(SQ_t)(x,y)| \leq C_{m}\frac{1}{\sqrt{1+t}}\vartheta(x,y)^{-4-|\alpha|} \]
for all $\lvert\alpha\rvert=\lvert\alpha_1\rvert+\lvert\alpha_2\rvert\leq m$.  Thus, $SQ_t$ is a smoothing operator of order $0$ with size $\frac{1}{\sqrt{1+t}}$. 

Arguing similarly yields that $SP_t$ and $SR_t$ are smoothing operators of orders $0$ with sizes $\frac{1}{1+t}$ and $1$, respectively.
\end{proof}

We need two results about the smoothing properties of the operators $G_t$ from Theorem~\ref{t-gue140816}.

\begin{lem}\label{l-gue141227}
Let $G_t\in L^{-2}_{{\rm cl\,},2}(M,\bR_+)$ be as in Theorem~\ref{t-gue140816}. Then, $\tau G_t\tau$ is a smoothing operator of order $2$ with size $\frac{1}{\sqrt{1+t}}$. Moreover, $\tau G_t\tau$ is also a smoothing operators of order $0$ with size $\frac{1}{1+t}$.
\end{lem}

\begin{proof}
From Theorem~\ref{t-gue140812}, Lemma~\ref{l-gue140813I} and \eqref{e-gue141204aI}, it is straightforward to see that 
\begin{equation}\label{e-gue141227a}
\begin{split}
\tau G_t\tau&=SG_tS+\ol SG_t\ol S+F_t\\
&=S\ol N\,\ol\Box_bG_tS+\ol SN\Box_bG_t\ol S+H_t,
\end{split}
\end{equation}
where $F_t$ and $H_t$ are  smoothing operators on $M$ depending smoothly on $t$ with the property that for all $m\in\bN_0$, there is a constant $C_m>0$ such that for all $t\in\bR_+$, 
\begin{equation}\label{e-gue141227aI}
\begin{split}
&\abs{F_t(x,y)}_{C^m(M\times M)}\leq C_m\frac{1}{1+t},\\
&\abs{H_t(x,y)}_{C^m(M\times M)}\leq C_m\frac{1}{1+t}.
\end{split}
\end{equation}
Note that $\Box_bG_t,\ol\Box_bG_t\in L^{0}_{{\rm cl\,},2}(M,\bR_+)$. From this observation, Theorem~\ref{thm2.2}, Theorem~\ref{t-gue141227} and \eqref{e-gue141227a}, we conclude that $\tau G_t\tau$ is a smoothing operator of order $0$ with size $\frac{1}{1+t}$. 

From Lemma~\ref{l-gue140813}, we have 
\begin{equation}\label{e-gue141227aII}
\begin{split}
S\ol N\,\ol\Box_bG_tS&=S\ol N\Box_bG_tS+S\ol NEG_tS\\
&=S\ol N[\Box_b,G_t]S+S\ol NEG_tS,
\end{split}
\end{equation}
where $E$ is a first order partial differential operator. Note that $[\Box_b,G_t], EG_t\in L^{-1}_{{\rm cl\,},2}(M,\bR_+)$. From this observation, Theorem~\ref{thm2.2} and Theorem~\ref{t-gue141227}, we conclude that $S\ol N[\Box_b,G_t]S+S\ol NEG_tS$ is a smoothing operator of order $2$ with size $\frac{1}{\sqrt{1+t}}$. 
Similarly, $\ol SN\Box_bG_t\ol S$ is a smoothing operator of order $2$ with  size $\frac{1}{\sqrt{1+t}}$. From \eqref{e-gue141227a}, we conclude that $\tau G_t\tau$ is a smoothing operator of order $2$
with size $\frac{1}{\sqrt{1+t}}$. The lemma follows.
\end{proof}

\begin{lem}\label{l-gue141227I}
Let $E_2\in L^2_{{\rm cl\,}}(M)$ be as in~\eqref{e-gue141204a}.  Then $\tau E_2(I-\tau)G_t\tau$ is a smoothing operator of order $1$ with size $1$.
\end{lem}

\begin{proof}
From Theorem~\ref{t-gue140812}, Lemma~\ref{l-gue140813I} and \eqref{e-gue141204aI}, we check that 
\begin{equation}\label{e-gue141227b}
\begin{split}
\tau E_2(I-\tau)G_t\tau&=SE_2(I-S)G_tS+\ol SE_2(I-\ol S)G_t\ol S+F_t\\
&=SE_2\Box_bNG_tS+\ol SE_2\ol\Box_b\,\ol NG_t\ol S+F_t,
\end{split}
\end{equation}
where $F_t$ is a smoothing operators on $M$ depending smoothly on $t$ with the property that for all $m\in\bN_0$, there is a constant $C_m>0$ such that for all $t\in\bR_+$, 
\[\abs{F_t(x,y)}_{C^m(M\times M)}\leq C_m\frac{1}{1+t}.\]
Again, from Theorem~\ref{t-gue140812}, Lemma~\ref{l-gue140813I} and \eqref{e-gue141204aI}, we check that 
\begin{equation}\label{e-gue141227bI-I}
\begin{split}
SE_2\Box_bNG_tS&=S[E_2,\ol{\pr}^*_b]\ddbar_bNG_tS\\
&=S[E_2,\ol{\pr}^*_b]\ddbar_bN\ol N^2\,\ol\Box_b^2G_tS+H_t,
\end{split}
\end{equation}
where $H_t$ is a smoothing operator on $M$ depending smoothly on $t$ with the property that for all $m\in\bN_0$, there is a constant $C_m>0$ such that for all $t\in\bR_+$, 
\[\abs{H_t(x,y)}_{C^m(M\times M)}\leq C_m\frac{1}{1+t}.\]
From Lemma~\ref{l-gue140813}, we have 
\begin{equation}\label{e-gue141227bI}
\begin{split}
&S[E_2,\ol{\pr}^*_b]\ddbar_bN\ol N^2\,\ol\Box_b^2G_tS\\
&=S[E_2,\ol{\pr}^*_b]\ddbar_bN\ol N^2\ol\Box_b[\Box_b,G_t]S
+S[E_2,\ol{\pr}^*_b]\ddbar_bN\ol N^2\ol\Box_bZ_0S\\
&=S[E_2,\ol{\pr}^*_b]\ddbar_bN\ol N^2[\Box_b,[\Box_b,G_t]]S+S[E_2,\ol{\pr}^*_b]\ddbar_bN\ol N^2Z_1[\Box_b,G_t]S\\
&\quad+S[E_2,\ol{\pr}^*_b]\ddbar_bN\ol N^2[\Box_b,Z_0]G_tS+S[E_2,\ol{\pr}^*_b]\ddbar_bN\ol N^2Z_2Z_0G_tS,
\end{split}
\end{equation}
where $Z_0, Z_1, Z_2$ are first order partial differential operators. Note that 
\[[\Box_b,[\Box_b,G_t]], Z_1[\Box_b,G_t], [\Box_b,Z_0]G_t, Z_2Z_0G_t\in L^{0}_{{\rm cl\,},2}(M,\bR_+).\]
From this observation and Theorem~\ref{t-gue141227}, we deduce that $[\Box_b,[\Box_b,G_t]]S$, $Z_1[\Box_b,G_t]S$, $[\Box_b,Z_0]G_tS$ and $Z_2Z_0G_tS$ are smoothing operators of order $0$ with sizes $1$. Moreover, from the symbolic calculus of Stein--Yung~\cite{SteinYung13}, we check that $S[E_2,\ol{\pr}^*_b]\ddbar_bN\ol N^2$ is a smoothing operator of order $1$. 

From the discussion above and \eqref{e-gue141227bI}, we conclude that $S[E_2,\ol{\pr}^*_b]\ddbar_bN\ol N^2\,\ol\Box_b^2G_tS$ is a smoothing operator of order $1$ with size $1$. Similarly, we can repeat the procedure above and conclude that $\ol SE_2(I-\ol S)G_t\ol S$ is a smoothing operator of order $1$ with size $1$. The lemma now follows from~\eqref{e-gue141227b}. 
\end{proof}

Our first goal is to invert $\oP_4^\prime+t+\pi$.  We begin by constructing a parametrix.

\begin{prop}\label{p-gue141228}
For every $N>0$, there are continuous operators 
\begin{align*}
&A_{N,t}=O(\frac{1}{1+t})\colon H^s(M)\To H^{s+\frac{1}{2}}(M) \quad\text{for all $s\in\bZ$}, \\
&R_{N,t}=O(\frac{1}{1+t})\colon H^s(M)\To H^{s+N}(M) \quad\text{for all $s\in\bZ$}
\end{align*}
depending continuously on $t$ such that
\begin{enumerate}
 \item $A_{K,t}$ is a smoothing operator of order $3$ with size $\frac{1}{\sqrt{1+t}}$;
 \item $A_{K,t}$ is a smoothing operator of order $1$ with size $\frac{1}{1+t}$;
 \item $(\oP_4^\prime+t+\pi)(\tau G_t\tau+\tau A_{K,t}\tau)=\tau+\tau R_{K,t}\tau$ on $\mP$.
\end{enumerate}
\end{prop}

\begin{proof}
From Theorem~\ref{t-gue140814I} and Theorem~\ref{t-gue140816}, we have 
\begin{equation}\label{e-gue141228b}
\begin{split}
(\oP_4^\prime+t+\pi)(\tau G_t\tau)&=\tau(E_2+t)\tau G_t\tau+\pi\tau G_t\tau\\
&=\tau(E_2+t)G_t\tau-\tau E_2(I-\tau)G_t\tau+\pi\tau G_t\tau\\
&=I+\tau A_t\tau\ \ \mbox{on $\mathcal{P}$},
\end{split}
\end{equation}
where 
\begin{equation}\label{e-gue141228bI}
A_t=-\tau E_2(I-\tau)G_t\tau+\tau\Td F_t\tau.
\end{equation}
Here $\Td F_t$ is a smoothing operator on $M$ depending smoothly on $t$ with the property that for all $m\in\mathbb N_0$, there is a constant $C_m>0$ such that for all $t\in\bR_+$, 
\begin{equation}\label{e-gue141228bII-I}
\abs{\Td F_t(x,y)}_{C^m(M\times M)}\leq C_m\frac{1}{1+t}.
\end{equation}
By Lemma~\ref{l-gue141227I}, we have that $A_t$ is a smoothing operator of order $1$ with size $1$. We claim that 
\begin{equation}\label{e-gue141228bII}
A_t=O(1):H^s(M)\To H^{s+\frac{1}{2}}(M) \quad\text{for all $s\in\bZ$}.
\end{equation}
From~\eqref{e-gue141227bI-I} and~\eqref{e-gue141227bI} we see that
\begin{equation}\label{e-gue141228bV}
\begin{split}
&SE_2\Box_bNG_tS\\
&=S[E_2,\ol{\pr}^*_b]\ddbar_bN\ol N^2[\Box_b,[\Box_b,G_t]]S+S[E_2,\ol{\pr}^*_b]\ddbar_bN\ol N^2Z_1[\Box_b,G_t]S\\
&\quad+S[E_2,\ol{\pr}^*_b]\ddbar_bN\ol N^2[\Box_b,Z_0]G_tS+S[E_2,\ol{\pr}^*_b]\ddbar_bN\ol N^2Z_2Z_0G_tS+H_t,
\end{split}
\end{equation}
where $Z_0, Z_1, Z_2$ are first order partial differential operators and $H_t$ is a smoothing operator on $M$ depending smoothly on $t$ with the property that for all $m\in\bN_0$, there is a constant $C_m>0$ such that for all $t\in\bR_+$, 
\[\abs{H_t(x,y)}_{C^m(M\times M)}\leq C_m\frac{1}{1+t}.\]
It is known that (see~\cite{Hsiao2010,Hsiao2014}) 
\begin{align*}
&N,\ol N\colon H^s(M)\To H^{s+1}(M) \quad\text{for all $s\in\bZ$},\\
&\ddbar_bN\colon H^s(M)\To H^{s+\frac{1}{2}}(M,T^{*0,1}M) \quad\text{for all $s\in\bZ$}.
\end{align*}
From this observation, \eqref{e-gue140816} and \eqref{e-gue141228bV}, we deduce that 
\[ SE_2\Box_bNG_tS=O(1)\colon H^s(M)\To H^{s+\frac{1}{2}}(M) \quad\text{for all $s\in\bZ$} . \]
Similarly, we have $\ol SE_2\ol\Box_b\,\ol NG_t\ol S=O(1):H^s(M)\To H^{s+\frac{1}{2}}(M)$ for all $s\in\bZ$.  Inserting this into~\eqref{e-gue141227b} yields the claim~\eqref{e-gue141228bII}.

Now put 
\[A_{K,t}=\tau G_t\tau(I-(\tau A_t\tau)+(\tau A_t\tau)^2-(\tau A_t\tau)^3+\cdots+(\tau A_t\tau)^{2K+4})-\tau G_t\tau.\]
From Theorem~\ref{thm2.2} and Lemma~\ref{l-gue141227} we observe that $A_{K,t}$ is a smoothing operator of order $3$ with size $\frac{1}{\sqrt{1+t}}$, and also $A_{K,t}$ is a smooth operator of order $1$ with size $\frac{1}{1+t}$. Moreover, from~\eqref{e-gue141204aI} and~\eqref{e-gue141228bII} we conclude that 
\[A_{K,t}=O(\frac{1}{1+t})\colon H^s(M)\To H^{s+\frac{1}{2}}(M) \]
for all $s\in\bZ$.  Furthermore, from \eqref{e-gue141228b}, we observe that
\begin{equation}\label{e-gue141228bVIII}
(\oP_4^\prime+t+\pi)(\tau G_t\tau+\tau A_{K,t}\tau)=\tau+(\tau A_t\tau)^{2K+5}.
\end{equation}
From \eqref{e-gue141228bII}, we see that
\[(\tau A_t\tau)^{2K+4}=O(1)\colon H^s(M)\To H^{s+K+2}(M) \]
for all $s\in\bZ$.  Moreover, from \eqref{e-gue141204aI} and \eqref{e-gue141228bI}, we observe that
\[\tau A_t\tau=O(\frac{1}{1+t})\colon H^s(M)\To H^{s-2}(M) \]
for all $s\in\bZ$.  Thus, $(\tau A_t\tau)^{2K+5}=O(\frac{1}{1+t})\colon H^s(M)\To H^{s+K}(M)$ for all $s\in\bZ$.  Combining this with~\eqref{e-gue141228bVIII} yields the result.
\end{proof}

\begin{rem}\label{r-gue150101}
It is easy to see that $A_{K,t}$ depends continuously on $t$ in $L^2(M)$.
\end{rem}

From now on, we identify the operator $(\oP_4^\prime+t+\pi)^{-1}\colon\hat{\mathcal{P}}\To\hat{\mathcal{P}}$ with $\tau(\oP_4^\prime+t+\pi)^{-1}\tau$. Thus $(\oP_4^\prime+t+\pi)^{-1}\colon L^2(M)\To L^2(M)$. We can extend and identify this operator as follows.

\begin{prop}\label{p-gue141229}
$(\oP_4^\prime+t+\pi)^{-1}$ can be continuously extended to $(\oP_4^\prime+t+\pi)^{-1}\colon H^s(M)\To H^s(M)$ for every $s\in\mathbb Z$.  Moreover, for every $K\in\mathbb N_0$ we have 
\[(\oP_4^\prime+t+\pi)^{-1}-(\tau G_t\tau+\tau A_{2K,t}\tau)=O(\frac{1}{1+t})\colon H^{-K}(M)\To H^K(M),\]
where $A_{2K,t}$ is as in Proposition~\ref{p-gue141228}. 
\end{prop}

\begin{proof}
Fix $K\in\mathbb N_0$ and let $A_{2K,t}$ and $R_{2K,t}$ be as in Proposition~\ref{p-gue141228}. Then 
\begin{equation}\label{e-gue141229}
\tau G_t\tau+\tau A_{2K,t}\tau=(\oP_4^\prime+t+\pi)^{-1}+(\oP_4^\prime+t+\pi)^{-1}\tau R_{2K,t}\tau.
\end{equation}
Note that $\tau R_{2K,t}\tau=O(\frac{1}{1+t})\colon H^{-s}(M)\To L^2(M)$ for all $s\in\bZ$ with $\abs{s}\leq 2K$.  By~\eqref{e-gue141204aI}, $\tau G_t\tau+\tau A_{2K,t}\tau=O(\frac{1}{1+t})\colon H^s(M)\To H^{s}(M)$ for all $s\in\bZ$.  By~\eqref{e-gue140816f}, we observe that $(\oP_4^\prime+t+\pi)^{-1}=O(\frac{1}{1+t})\colon L^2(M)\To L^2(M)$. From these observations we conclude that we can extend to $(\oP_4^\prime+t+\pi)^{-1}$ to $H^{-s}(M)$ for all $s\in\bN_0$ with $s\leq 2K$; indeed
\begin{equation}\label{e-gue141229I}
(\oP_4^\prime+t+\pi)^{-1}=O(\frac{1}{1+t})\colon H^{-s}(M)\To H^{-s}(M)
\end{equation}
for all $s\in\bN_0$ with $s\leq 2K$.  By taking the adjoint in \eqref{e-gue141229I}, we conclude that we can extend to $(\oP_4^\prime+t+\pi)^{-1}$ to
\begin{equation}\label{e-gue141229II}
(\oP_4^\prime+t+\pi)^{-1}=O(\frac{1}{1+t})\colon H^{s}(M)\To H^{s}(M)
\end{equation}
for all $s\in\bN_0$ with $s\leq 2K$.  From~\eqref{e-gue141229} and~\eqref{e-gue141229II} we conclude that 
\[(\oP_4^\prime+t+\pi)^{-1}-(\tau G_t\tau+\tau A_{2K,t}\tau)=O(\frac{1}{1+t})\colon H^{-K}(M)\To H^K(M). \qedhere \]
\end{proof}

This allows us to prove the following theorem.

\begin{thm}\label{t-gue141229I}
There is a $G\in L_{{\rm cl\,}}^{-1}(M)$ such that $2GE_1-I\in L_{{\rm cl\,}}(M)$ for $E_1\in L_{{\rm cl\,}}^1(M)$ as in Theorem~\ref{t-gue140814I} and for every $\ell\in\bN_0$,
\[(\oP_4^\prime)^{-\frac{1}{2}}=\tau G\tau+\tau A_\ell\tau+\tau R_\ell\tau \]
on $\hat\mP$, where $A_\ell,R_\ell\colon C^\infty(M)\to\mathscr D'(M)$ are continuous operators, $R_\ell(x,y)\in C^\ell(M\times M)$, and $A_\ell$ is a smooth operator of order $3-\varepsilon$ for every $0<\varepsilon<1$.
\end{thm}

\begin{proof}
Fix $\ell\in\mathbb N_0$ and take $K\gg\ell$. Put 
\[\Xi_{2K,t}=(\oP_4^\prime+t+\pi)^{-1}-(\tau G_t\tau+\tau A_{2K,t}\tau),\]
where $A_{2K,t}$ is as in Proposition~\ref{p-gue141228}. By Proposition~\ref{p-gue141229}, $\Xi_{2K,t}$ is well-defined as a continuous operator $H^s(M)\To H^s(M)$ for every $s\in\mathbb Z$. Observe that $\Xi_{2K,t}=\tau\Xi_{2K,t}\tau$. From Lemma~\ref{l-gue140815} we see that 
\begin{equation}\label{e-gue141230I}
(\oP_4^\prime)^{-\frac{1}{2}}=c\int^\infty_0t^{-\frac{1}{2}}\tau G_t\tau dt+c\int^\infty_0t^{-\frac{1}{2}}\tau A_{2K,t}\tau dt+c\int^\infty_0t^{-\frac{1}{2}}\tau\Xi_{2K,t}\tau dt.
\end{equation}
It is known that (see~\cite{Shubin2001})
\begin{equation}\label{e-gue141230II}
c\int^\infty_0t^{-\frac{1}{2}}\tau G_t\tau dt=\tau G\tau,
\end{equation}
where $G\in L^{-1}_{{\rm cl\,}}(M)$ with $2GE_1-I\in L^{-1}_{{\rm cl\,}}(M)$.

We claim that 
\begin{equation}\label{e-gue141230a}
\Xi(x,y):=(c\int^\infty_0t^{-\frac{1}{2}}\Xi_{2K,t}dt)(x,y)\in C^\ell(M\times M)
\end{equation}
if $K$ is large enough.  Fix $k\in\mathbb N_0$. For every $m\in\mathbb N$, consider 
\[\Xi_{k,m}:=c\sum^m_{j=1}\frac{1}{m}\Xi_{2K,k+\frac{j}{m}}\frac{1}{\sqrt{k+\frac{j}{m}}}.\]
It is clear that, in $L^2(M)$, 
\begin{equation}\label{e-gue141230aI}
\lim_{m\To\infty}\Xi_{k,m}=c\int^{k+1}_kt^{-\frac{1}{2}}\Xi_{2K,t}dt .
\end{equation}
By Proposition~\ref{p-gue141229}, we see that 
\begin{equation}\label{e-gue141230aII}
\norm{\Xi_{k,m}}_{\mathscr L(H^{-K}(M), H^{K}(M))}\leq c_1\sum^m_{j=1}\frac{1}{m}\frac{1}{1+k+\frac{j}{m}}\frac{1}{\sqrt{k+\frac{j}{m}}}\leq c_1\int^{k+1}_k\frac{1}{(1+t)\sqrt{t}}dt,
\end{equation}
where $c_1>0$ is a constant and $\norm{\Xi_{k,m}}_{\mathscr L(H^{-K}(M), H^{K}(M))}$ denotes the standard operator norm of $\Xi_{k,m}$ in $\mathscr L(H^{-K}(M), H^{K}(M))$. From \eqref{e-gue141230aII} and the Sobolev embedding theorem, if $K\gg\ell$, there is a subsequence $(m_s)$ such that $m_s\to\infty$ as $s\to\infty$,
\begin{equation}\label{e-gue141230aIII}
\lim_{s\To\infty}\Xi_{k,m_s}(x,y)=\Xi_{k}(x,y)
\end{equation}
in $C^\ell(M\times M)$, and 
\begin{equation}\label{e-gue141230aIV}
\norm{\Xi_{k}(x,y)}_{C^\ell(M\times M)}\leq\Td c_1\int^{k+1}_k\frac{1}{(1+t)\sqrt{t}}dt,
\end{equation}
where $\Td c_1>0$ is a constant independent of $k$.
From \eqref{e-gue141230aI}, \eqref{e-gue141230aIII} and \eqref{e-gue141230aIV}, we conclude that 
\begin{equation}\label{e-gue141230aV}
\begin{split}
&\Xi_k(x,y)=(c\int^{k+1}_kt^{-\frac{1}{2}}\Xi_{2K,t}dt)(x,y)\in C^\ell(M\times M),\\
&\norm{c\int^{k+1}_kt^{-\frac{1}{2}}\Xi_{2K,t}dt)}_{C^\ell(M\times M)}\leq\Td c_1\int^{k+1}_k\frac{1}{(1+t)\sqrt{t}}dt.
\end{split}
\end{equation}
Since $\sum^\infty_{k=0}\int^{k+1}_k\frac{1}{(1+t)\sqrt{t}}dt=\int^\infty_0\frac{1}{(1+t)\sqrt{t}}dt<\infty$, we deduce that
\[\Xi(x,y)=(c\int^\infty_0t^{-\frac{1}{2}}\Xi_{2K,t}dt)(x,y)=\sum^\infty_{k=0}\Xi_k(x,y)\in C^\ell(M\times M) ,\]
as claimed. 

From now on, we take $K$ large enough so that $\Xi(x,y)\in C^\ell(M\times M)$.  Put $A:=c\int^\infty_0t^{-\frac{1}{2}}A_{2K,t}dt$. We now study the kernel of $A$. Fix $x_0, y_0\in M$ and set $\vartheta(x_0,y_0)=r$. Put 
\[B_{x_0}(\frac{r}{4})=\set{z\in M;\, \vartheta(z,x_0)<\frac{r}{4}},\ \ B_{y_0}(\frac{r}{4})=\set{z\in M;\, \vartheta(z,y_0)<\frac{r}{4}}.\] 
Take $\chi\in C^\infty_0(B_{x_0}(\frac{r}{4}))$ and $\chi_1\in C^\infty_0(B_{y_0}(\frac{r}{4}))$ such that $\chi=1$ near $x_0$ and $\chi_1=1$ near $y_0$. Consider $\Td A:=c\int^\infty_0t^{-\frac{1}{2}}\chi A_{2K,t}\chi_1dt$. Then, 
\begin{equation}\label{e-gue141230ab}
\Td A=c\int^{r^{-4}}_0t^{-\frac{1}{2}}\chi A_{2K,t}\chi_1dt+c\int^{\infty}_{r^{-4}}t^{-\frac{1}{2}}\chi A_{2K,t}\chi_1dt.
\end{equation}
For every $m\in\mathbb N$, consider 
\[B_m=c\sum^m_{j=1}\frac{r^{-4}}{m}(\chi A_{2K,\frac{j}{m}r^{-4}}\chi_1)(\frac{j}{m}r^{-4})^{-\frac{1}{2}}.\]
It is easy to see that, in $L^2$,
\begin{equation}\label{e-gue141230abI}
\lim_{m\To\infty}B_m=c\int^{r^{-4}}_0t^{-\frac{1}{2}}\chi A_{2K,t}\chi_1dt .
\end{equation}
Recall from Proposition~\ref{p-gue141228} that $A_{2K,t}$ is a smoothing operator of order $3$ with size $\frac{1}{\sqrt{1+t}}$. From this and \eqref{e-gue141230abI}, we have that, for any $x\in B_{x_0}(\frac{r}{4})$, $y\in B_{y_0}(\frac{r}{4})$, 
\begin{align*}
\abs{B_m(x,y)}&\leq c\sum^m_{j=1}\frac{r^{-4}}{m}\abs{(\chi A_{2K,\frac{r^{-4}j}{m}}\chi_1)(x,y)}(\frac{r^{-4}j}{m})^{-\frac{1}{2}}\\
&\leq c\sum^m_{j=1}\frac{r^{-4}}{m}\frac{1}{\sqrt{1+\frac{r^{-4}j}{m}}}\vartheta(x,y)^{-1}(\frac{r^{-4}j}{m})^{-\frac{1}{2}}\\
&\leq c_2\vartheta(x,y)^{-1}\int^{r^{-4}}_0\frac{1}{\sqrt{1+t}}t^{-\frac{1}{2}}dt\\
&\leq c_3\vartheta(x,y)^{-1}\abs{\log\vartheta(x,y)},
\end{align*}
where $c_2>0$, $c_3>0$ are constants independent of $m$, $r$, $\chi$, $\chi_1$, $x_0$, $y_0$. Similarly, for every $\alpha_1, \alpha_2\in\mathbb N_0$ and $\varepsilon>0$, there is a constant $C_{\alpha_1,\alpha_2,\varepsilon}$, independent of $m$, $r$, $\chi$, $\chi_1$, $x_0$, $y_0$, such that 
\begin{equation}\label{e-gue150101}
\abs{(\nabla_b)^{\alpha_1}_x(\nabla_b)^{\alpha_1}_yB_m(x,y)}\leq C_{\alpha_1,\alpha_2,\varepsilon}\vartheta(x,y)^{-1-\abs{\alpha_1}-\abs{\alpha_2}-\varepsilon}.
\end{equation}
From \eqref{e-gue150101}, we deduce that there is a subsequence $(m_s)$ such that $m_s\to\infty$ as $s\to\infty$ for which $B_{m_s}(x,y)$ converges to some $B(x,y)$ in the $C^\infty(M\times M)$ topology with the property that for every $\alpha_1,\alpha_2\in\bN_0$ and every $\varepsilon>0$, there is a constant $C_{\alpha_1,\alpha_2,\varepsilon}$, independent of $m,r,\chi,\chi_1,x_0,y_0$, such that
\begin{equation}\label{e-gue150101I}
\abs{(\nabla_b)^{\alpha_1}_x(\nabla_b)^{\alpha_2}_yB(x,y)}\leq C_{\alpha_1,\alpha_2,\varepsilon}\vartheta(x,y)^{-1-\abs{\alpha_1}-\abs{\alpha_2}-\varepsilon}.
\end{equation}
In particular, from~\eqref{e-gue141230abI} we have that
\begin{equation}\label{e-gue150101II-I}
(c\int^{r^{-4}}_0t^{-\frac{1}{2}}\chi A_{2K,t}\chi_1dt)(x,y)=B(x,y).
\end{equation} 

Fix $k\in\mathbb N_0$. For every $m\in\mathbb N$, consider 
\[D_{k,m}:=c\sum^m_{j=1}\frac{r^{-4}}{m}\chi A_{2K,r^{-4}k+\frac{j}{m}r^{-4}}\chi_1\frac{1}{\sqrt{r^{-4}k+\frac{j}{m}r^{-4}}}.\]
It is clear that 
\begin{equation}\label{e-gue141230abVI}
\lim_{m\To\infty}D_{k,m}=c\int^{r^{-4}(k+1)}_{r^{-4}k}t^{-\frac{1}{2}}\chi A_{2K,t}\chi_1dt
\end{equation}
in $L^2(M)$.  Recall from Proposition~\ref{p-gue141228} that $A_{2K,t}$ is a smoothing operator of order $1$ with size $\frac{1}{1+t}$. From this observation, we find that for every $x\in B_{x_0}(\frac{r}{4})$, $y\in B_{y_0}(\frac{r}{4})$, we have that
\begin{equation}\label{e-gue141231I}
\begin{split}
\abs{D_{k,m}(x,y)}&\leq\Td c_1\sum^m_{j=1}\frac{r^{-4}}{m}\abs{(\chi A_{2K,r^{-4}k+\frac{j}{m}r^{-4}}\chi_1)(x,y)}(r^{-4}k+\frac{j}{m}r^{-4})^{-\frac{1}{2}}\\
&\leq\Td c_2\sum^m_{j=1}\frac{r^{-4}}{m}\vartheta(x,y)^{-3}\frac{1}{1+r^{-4}k+\frac{j}{m}r^{-4}}(r^{-4}k+\frac{j}{m}r^{-4})^{-\frac{1}{2}}\\
&\leq\Td c_2\int^{r^{-4}(k+1)}_{r^{-4}k}\vartheta(x,y)^{-3}\frac{1}{1+t}\frac{1}{\sqrt{t}}dt,
\end{split}
\end{equation}
where $\Td c_1>0$, $\Td c_2>0$ are constants independent of $k$, $m$, $r$, $\chi$, $\chi_1$, $x_0$, $y_0$. Similarly, for every $\alpha_1, \alpha_2\in\mathbb N_0$, there is a constant $C_{\alpha_1,\alpha_2}$, independent of $m,k,r,\chi,\chi_1,x_0,y_0$, such that 
\[ \abs{(\nabla_b)^{\alpha_1}_x(\nabla_b)^{\alpha_1}_yD_{k,m}(x,y)}\leq C_{\alpha_1,\alpha_2}\int^{r^{-4}(k+1)}_{r^{-4}k}\vartheta(x,y)^{-3-\abs{\alpha_1}-\abs{\alpha_2}}\frac{1}{1+t}\frac{1}{\sqrt{t}}dt. \]
Therefore there is a subsequence $(m_s)$ such that $m_s\to s$ as $s\to\infty$ for which $D_{k,m_s}(x,y)$ converges to some $D_k(x,y)$ in the $C^\infty(M\times M)$ topology with the property that for every $\alpha_1,\alpha_2\in\bN_0$ and every $\varepsilon>0$, there is a constant $\Td C_{\alpha_1,\alpha_2,\varepsilon}$ such that
\begin{equation}\label{e-gue150101IV}
\abs{(\nabla_b)^{\alpha_1}_x(\nabla_b)^{\alpha_2}_yD_k(x,y)}\leq\Td C_{\alpha_1,\alpha_2,\varepsilon}\int^{r^{-4}(k+1)}_{r^{-4}k}\vartheta(x,y)^{-3-\abs{\alpha_1}-\abs{\alpha_2}-\varepsilon}\frac{1}{1+t}\frac{1}{\sqrt{t}}dt.
\end{equation}
In particular, from~\eqref{e-gue141230abVI} we find that 
\begin{equation}\label{e-gue150101V}
(c\int^{r^{-4}(k+1)}_{r^{-4}k}t^{-\frac{1}{2}}\chi A_{2K,t}\chi_1dt)(x,y)=D_k(x,y).
\end{equation} 
Note that for $x\in B_{x_0}(\frac{r}{4})$ and $y\in B_{y_0}(\frac{r}{4})$, 
\begin{equation}\label{e-gue150101VI}
\sum^\infty_{k=1}\int^{r^{-4}(k+1)}_{r^{-4}k}\vartheta(x,y)^{-3-\abs{\alpha_1}-\abs{\alpha_2}-\varepsilon}\frac{1}{1+t}\frac{1}{\sqrt{t}}dt 
\leq\hat c_0\vartheta(x,y)^{-1-\abs{\alpha_1}-\abs{\alpha_2}-\varepsilon},
\end{equation}
where $\hat c_0>0$ is a constant. From~\eqref{e-gue150101IV}, \eqref{e-gue150101V}, and~\eqref{e-gue150101VI} we deduce that 
$(c\int^{\infty}_{r^{-4}}t^{-\frac{1}{2}}\chi A_{2K,t}\chi_1dt)(x,y)\in C^\infty(M\times M)$ and for every $\alpha_1, \alpha_2\in\mathbb N_0$ and $\varepsilon>0$, there is a constant $\hat C_{\alpha_1,\alpha_2,\varepsilon}$, independent of $r$, $\chi$, $\chi_1$, $x_0$, $y_0$, such that 
\begin{equation}\label{e-gue150101VII}
\abs{(\nabla_b)^{\alpha_1}_x(\nabla_b)^{\alpha_2}_y\Bigr(c\int^{\infty}_{r^{-4}}t^{-\frac{1}{2}}\chi A_{2K,t}\chi_1dt\Bigr)(x,y)}\leq\hat C_{\alpha_1,\alpha_2,\varepsilon}\vartheta(x,y)^{-1-\abs{\alpha_1}-\abs{\alpha_2}-\varepsilon}.
\end{equation}

From \eqref{e-gue141230ab}, \eqref{e-gue150101I}, \eqref{e-gue150101II-I} and \eqref{e-gue150101VII}, we deduce that 
$A(x,y)$ satisfies the following differential inequalities when $x\neq y$: For every $\varepsilon>0$ and every $\alpha_1, \alpha_2\in\mathbb N_0$, there is a constant $C_{\alpha_1,\alpha_2,\varepsilon}>0$ independent of $x$ and $y$ such that
\begin{equation}\label{e-gue141230}
|(\nabla_b)_x^{\alpha_1} (\nabla_b)_y^{\alpha_2}A(x,y)| \leq C_{\alpha_1,\alpha_2,\varepsilon}\vartheta(x,y)^{-1-\varepsilon-|\alpha|}
\end{equation}
for all $\lvert\alpha\rvert=\lvert\alpha_1\rvert+\lvert\alpha_2\rvert$.

Now, we prove that $A$ satisfies the cancellation condition of order $3-\varepsilon$ for every $0<\varepsilon<1$. Let $\phi$ be a normalized bump function in $B(x,r)$. Then 
\begin{equation}\label{e-gue150117}
\begin{split}
&\norm{\nabla^\alpha_bA\phi}_{L^\infty(B(x,r))}\\
&\leq c\int^{r^{-4}}_0t^{-\frac{1}{2}}\norm{\nabla^\alpha_bA_{2K,t}\phi}_{L^{\infty}(B(x,r))}+c\int^{\infty}_{r^{-4}}t^{-\frac{1}{2}}\norm{\nabla^\alpha_bA_{2K,t}\phi}_{L^{\infty}(B(x,r))}.
\end{split}
\end{equation}
Since $A_{2K,t}$ is a smoothing operator of order $3$ with size $\frac{1}{\sqrt{1+t}}$, we have that
\begin{equation}\label{e-gue150117I}
\int^{r^{-4}}_0t^{-\frac{1}{2}}\norm{\nabla^\alpha_bA_{2K,t}\phi}_{L^{\infty}(B(x,r))} \leq c_2r^{3-\abs{\alpha}}\abs{\log r},
\end{equation}
where $c_1>0$ and $c_2>0$ are constants independent of $r$. Since $A_{2K,t}$ is also a smoothing operator of order $1$ with size $\frac{1}{1+t}$, we have 
\begin{equation}\label{e-gue150117II}
\int^{\infty}_{r^{-4}}t^{-\frac{1}{2}}\norm{\nabla^\alpha_bA_{2K,t}\phi}_{L^{\infty}(B(x,r))}\leq \hat c_1r^{1-\abs{\alpha}}\int^{\infty}_{r^{-4}}t^{-\frac{1}{2}}\frac{1}{1+t}dt\leq \hat c_2r^{3-\abs{\alpha}},
\end{equation}
where $\hat c_1>0$ and $\hat c_2>0$ are constants independent of $r$. From \eqref{e-gue150117}, \eqref{e-gue150117I} and \eqref{e-gue150117II}, we deduce that $A$ satisfies the cancellation condition of order $3-\varepsilon$ for every $0<\varepsilon<1$. Similarly, $A^*$ satisfies the cancellation condition of order $3-\varepsilon$ for every $0<\varepsilon<1$. The conclusion follows from~\eqref{e-gue141230}.
\end{proof}

Now let us consider $2\Delta_b+\frac{1}{2}R$ extended to $L^2(M)$ in the standard way. Since $2\Delta_b+\frac{1}{2}R$ is hypoelliptic with loss of one derivative, 
\[2\Delta_b+\frac{1}{2}R\colon{\rm Dom\,}(2\Delta_b+\frac{1}{2}R)\to L^2(M)\]
has closed range, is self-adjoint, and ${\rm Ker\,}(2\Delta_b+\frac{1}{2}R)$ is a finite-dimensional subspace of $C^\infty(M)$. Let $\hat N\colon L^2(M)\to{\rm Dom\,}(2\Delta_b+\frac{1}{2}R)$ be the partial inverse and let $p\colon L^2(M)\to{\rm Ker\,}(2\Delta_b+\frac{1}{2}R)$ be the orthogonal projection. Then $p$ is a smoothing operator on $M$ and we have 
\[
\begin{split}
&\hat N(2\Delta_b+\frac{1}{2}R)+p=I\ \ \mbox{on ${\rm Dom\,}(2\Delta_b+\frac{1}{2}R)$},\\
&(2\Delta_b+\frac{1}{2}R)\hat N+p=I\ \ \mbox{on $L^2(M)$}.
\end{split}\]
Note that $\hat N\colon H^s(M)\to H^{s+1}(M)$ for all $z\in\bZ$.  Moreover, $\hat N$ is a smoothing operator of order $2$ (see~\cite[Section~10]{HsiaoYung2014} and~\cite{SteinYung13}).

\begin{prop}\label{p-gue150101}
With the notations above, $\tau G\tau-\tau\hat N$ is a smoothing operator of order $3$, where $G\in L^{-1}_{{\rm cl\,}}(M)$ is as in Theorem~\ref{t-gue141229I}.
\end{prop} 

\begin{proof}
From Theorem~\ref{t-gue140813} and Theorem~\ref{t-gue141229I}, we have that
\begin{equation}\label{e-gue150101f}
(\tau G\tau)(\tau(2\Delta_b+\frac{R}{2})\tau)=\tau+\tau P_{-1}\tau-\tau G(I-\tau)(2E_1+\frac{R}{2})\tau,
\end{equation}
where $P_{-1}\in L^{-1}_{{\rm cl\,}}(M)$.

We claim that
\begin{align}
\label{e-gue150101fI} &\mbox{$\tau P_{-1}\tau$ is a smoothing operator of order $2$}, \\
\label{e-gue150101fII} &\mbox{$\tau G(I-\tau)(2E_1+\frac{R}{2})\tau$ is a smoothing operator of order $1$}.
\end{align}
From Theorem~\ref{t-gue140812} and Lemma~\ref{l-gue140813I}, we have that
\begin{equation}\label{e-gue150101fIII}
\tau P_{-1}\tau\equiv SP_{-1}S+\ol SP_{-1}\ol S.
\end{equation}
From Lemma~\ref{l-gue140813} and Lemma~\ref{l-gue140813I} we have that
\begin{equation}\label{e-gue150101fIV}
\begin{split}
SP_{-1}S\equiv S\ol N\,\ol\Box_bP_{-1}S&=S\ol N\Box_bP_{-1}S+S\ol NL_1P_{-1}S\\
&=S\ol N[\Box_b,P_{-1}]S+S\ol NL_1P_{-1}S,
\end{split}
\end{equation}
where $L_1$ is a first order partial differential operator. From the symbolic calculus of Stein--Yung~\cite{SteinYung13}, we check that $[\Box_b,P_{-1}]S$ and $L_1P_{-1}S$ are smoothing operators of order $0$. 
From this observation, \eqref{e-gue150101fIV} and Theorem~\ref{thm2.2}, we conclude that $SP_{-1}S$ is a smoothing operator of order $2$. Similarly, $\ol SP_{-1}\ol S$ is a smoothing operator of order $2$. From \eqref{e-gue150101fIII}, we obtain~\eqref{e-gue150101fI}. 

Again, from Theorem~\ref{t-gue140812}, Lemma~\ref{l-gue140813I} and Lemma~\ref{l-gue140813}, we have that
\begin{equation}\label{e-gue150101fV}
\begin{split}
&\tau G(I-\tau)(2E_1+\frac{R}{2})\tau\\
&\equiv SG(I-S)(2E_1+\frac{R}{2})S+\ol SG(I-\ol S)(2E_1+\frac{R}{2})\ol S\\
&\equiv SG\ol\Box_b\,\ol NN\Box_b(2E_1+\frac{R}{2})S+\ol SG\Box_bN\ol N\,\ol\Box_b(2E_1+\frac{R}{2})\ol S\\
&=S[G,\Box_b]\ol NN\ol{\pr}^*_b[\ddbar_b,2E_1+\frac{R}{2}]S+SGL_1\ol NN\ol{\pr}^*_b[\ddbar_b,2E_1+\frac{R}{2}]S\\
&\quad+\ol S[G,\ol \Box_b]N\ol N\pr^*_b[\pr_b,2E_1+\frac{R}{2}]\ol S+\ol SG\ol L_1N\ol N\pr^*_b[\pr_b,2E_1+\frac{R}{2}]\ol S,
\end{split}
\end{equation}
where $L_1$ is a first order partial differential operator. From the symbolic calculus of Stein--Yung~\cite{SteinYung13}, we check that $\ol NN\ol{\pr}^*_b[\ddbar_b,2E_1+\frac{R}{2}]S$, $N\ol N\pr^*_b[\pr_b,2E_1+\frac{R}{2}]\ol S$ are smoothing operators of order $1$ and $S[G,\Box_b]$, $SGL_1$, $\ol S[G,\ol\Box_b]$, $\ol SG\ol L_1$ are smoothing operators of order $0$. From this observation, \eqref{e-gue150101fV} and Theorem~\ref{thm2.2}, we obtain~\eqref{e-gue150101fII}. 

Now, from Theorem~\ref{t-gue140812}, Lemma~\ref{l-gue140813I}, Lemma~\ref{l-gue140813} and recall that $\Delta_b=\Box_b+\ol\Box_b$, we have that
\begin{equation}\label{e-gue150102}
\begin{split}
(\tau(2\Delta_b+\frac{R}{2})\tau)\hat N&\equiv\tau-\tau(2\Delta_b+\frac{R}{2})(I-\tau)\hat N\\
&\equiv\tau-S(2\Delta_b+\frac{R}{2})(I-S)\hat N-\ol S(2\Delta_b+\frac{R}{2})(I-\ol S)\hat N\\
&=\tau-S(\ol\Box_b+\frac{R}{2})\Box_bN\hat N-\ol S(\Box_b+\frac{R}{2})\ol\Box_b\,\ol N\hat N\\
&\equiv\tau-S[L_1+\frac{R}{2},\ol{\pr}^*_b]\ddbar_bN\hat N-\ol S[\ol L_1+\frac{R}{2},\pr^*_b]\pr_b\ol N\hat N,
\end{split}
\end{equation}
where $L_1$ is a first order partial differential operator. From the symbolic calculus of Stein--Yung~\cite{SteinYung13}, we check that $[L_1+\frac{R}{2},\ol{\pr}^*_b]\ddbar_bN\hat N$ and $[\ol L_1+\frac{R}{2},\pr^*_b]\pr_b\ol N\hat N$ are smoothing operators of order $1$. From this observation, \eqref{e-gue150102} and Theorem~\ref{thm2.2}, we obtain that 
\begin{equation}\label{e-gue150102I}
(\tau(2\Delta_b+\frac{R}{2})\tau)\hat N=\tau+H,
\end{equation}
where $H$ is a smoothing operator of order $1$. From \eqref{e-gue150101f} and \eqref{e-gue150102I}, we find that 
\begin{equation}\label{e-gue150102II}
\tau G\tau+(\tau G\tau)H=\tau\hat N+(\tau P_{-1}\tau)\hat N-\tau G(I-\tau)(2E_1+\frac{R}{2})\tau\hat N.
\end{equation}
We can repeat the proof of \eqref{e-gue150101fI} and deduce that $\tau G\tau$ is a smoothing operator of order $2$ and hence 
\begin{equation}\label{e-gue150102III}
\mbox{$(\tau G\tau)H$ is a smoothing operator of order $3$.}
\end{equation}
From~\eqref{e-gue150101fI}, \eqref{e-gue150101fII}, \eqref{e-gue150102II} and~\eqref{e-gue150102III} we deduce that $\tau G\tau-\tau\hat N$ is a smoothing operator of order $3$.
\end{proof}

Fix a point $\zeta\in X$. The Green's function of $(\oP_4^\prime)^{\frac{1}{2}}$ at $\zeta$ is given by 
\begin{equation}\label{e-gue150106b}
G_\zeta:=(\oP_4^\prime)^{-\frac{1}{2}}\tau\delta_\zeta\tau\in\mathscr D'(M).
\end{equation}
It is easy to see that 
\begin{equation}\label{e-gue140827I}
\mbox{$(\oP_4^\prime)^{\frac{1}{2}}G_\zeta=\delta_\zeta-\pi(x,\zeta)$ on $\mathcal{P}$}.
\end{equation}
Note that $\pi(x,\zeta)\in C^\infty(M)\cap{\rm Ker\,}(\oP_4^\prime)^{-\frac{1}{2}}$. 

\begin{proof}[Proof of Theorem~\ref{t-gue140827}]
Fix $\zeta\in M$ and let $(z,t)$ be CR normal coordinates defined in a neighborhood of $\zeta$ such that $(z(\zeta),t(\zeta))=(0,0)$. For $m\in\bR$, let $\mE(\rho^m)$ be as in the discussion before Theorem~\ref{t-gue140827}.
Let $\ell_0\in\mathbb N_0$ and fix $\ell\gg\ell_0$. From Theorem~\ref{t-gue141229I} and Proposition~\ref{p-gue150101}, we have 
\begin{equation}\label{e-gue150106bI}
\begin{split}
G_\zeta&=\tau G\tau\delta_\zeta\tau+\tau A_\ell\tau\delta_\zeta\tau+\tau R_\ell\tau\delta_\zeta\tau\\
&=\tau\hat N\delta_\zeta\tau+\tau K\delta_\zeta\tau+\tau A_\ell\tau\delta_\zeta\tau+\tau R_\ell\tau\delta_\zeta\tau,
\end{split}
\end{equation}
where $K$ is a smoothing operator of order $3$. Since $R_\ell(x,y)\in C^\ell(M\times M)$, we can take $\ell$ large enough so that 
\begin{equation}\label{e-gue150106bII}
R_\ell\tau\delta_\zeta\in C^{\ell_0}(M).
\end{equation}
Since $K$ is a smoothing operator of order $3$,
\begin{equation}\label{e-gue150107I}
K\delta_\zeta\in\mathcal{E}(\rho^{-1}).
\end{equation}
From Theorem~\ref{thm2.2} and Theorem~\ref{t-gue141229I} we see that $A_\ell\tau$ is a smoothing operator of order $3-\varepsilon$, for every $0<\varepsilon<1$. Hence, 
\begin{equation}\label{e-gue150107}
A_\ell\tau\delta_\zeta\in\mathcal{E}(\rho^{-1-\varepsilon})
\end{equation}
for all $\varepsilon>0$.

Finally, we consider $\hat N\delta_\zeta$. It is clear that $\hat N\delta_\zeta$ is the Green's function of $2\Delta_b+\frac{1}{2}R$. It was shown in~\cite[Section~5]{ChengMalchiodiYang2013} that, near $\zeta$, $\hat N\delta_\zeta$ has the form
\begin{equation}\label{e-gue140827a}
\hat N\delta_\zeta(z,t)=\rho(z,t)^{-2}+\omega_0
\end{equation}
for some $\omega_0\in C^1(M)$. Moreover, repeating the method in~\cite[Section~10]{HsiaoYung2014}, we conclude that 
\begin{equation}\label{e-gue150107VI}
\omega_0\in\mathcal{E}(\rho^{-\varepsilon})
\end{equation}
for all $\varepsilon>0$.  The conclusion follows from~\eqref{e-gue150106bI}, \eqref{e-gue150106bII}, \eqref{e-gue150107I}, \eqref{e-gue150107}, \eqref{e-gue140827a} and \eqref{e-gue150107VI}. 
\end{proof}

In the proof of Theorem~\ref{thm:regularity}, we need the following result.

\begin{thm}\label{t-gue150107}
For every $\ell\in\mathbb N_0$, we have 
\[\tau B_\ell\tau\oP^\prime_4=\tau+\tau C_\ell\tau\ \ \mbox{on $\hat{\mathcal{P}}$},\]
where $B_\ell, C_\ell\colon C^\infty(M)\To\mathscr D'(M)$ are continuous operators, $B_\ell$ is a smoothing operator of order $4-\varepsilon$ for all $0<\varepsilon<1$, and $(\tau C_\ell\tau)(x,y)\in C^\ell(M\times M)$.
\end{thm} 

\begin{proof}
In view of \eqref{e-gue140811IV}, we see that $\oP^\prime_4=\tau(4\Delta^2_b+L_2)\tau$, where $L_2=\nabla^2_b+\nabla_b+r$, $r\in C^\infty(X)$. Let $H$ be a parametrix of $4\Delta^2_b+L_2$. Then $H\colon H^s(M)\To H^{s+2}(M)$ for every $s\in\mathbb Z$ and $H$ is a smoothing operator of order $4-\varepsilon$ for every $0<\varepsilon<1$. From Theorem~\ref{t-gue140812} and Lemma~\ref{l-gue140813I}, we have that 
\begin{equation}\label{e-gue150109}
\begin{split}
\tau H\tau\oP^\prime_4&=(\tau H\tau)(\tau(4\Delta^2_b+L_2)\tau)\\
&=\tau-\tau H(I-\tau)(4\Delta^2_b+L_2)\tau-F_0\\
&=\tau-SH(I-S)(4\Delta^2_b+L_2)S-\ol SH(I-\ol S)(4\Delta^2_b+L_2)\ol S-F_1,
\end{split}
\end{equation}
where $F_0$ and $F_1$ are smoothing operators on $M$. Put 
\begin{equation}\label{e-gue150109I}
\Upsilon=SH(I-S)(4\Delta^2_b+L_2)S+\ol SH(I-\ol S)(4\Delta^2_b+L_2)\ol S+F_1.
\end{equation}
Note that $\Upsilon=\tau\Upsilon\tau$.
Repeating the procedure in \eqref{e-gue150102}, we conclude that 
\begin{equation}\label{e-gue150109II}
\begin{split}
&SH(I-S)(4\Delta^2_b+L_2)S=SHN\ol{\pr}^*_bQ_2S,\\
&\ol SH(I-\ol S)(4\Delta^2_b+L_2)S=\ol SH\ol N\pr^*_b\Td Q_2\ol S,
\end{split}
\end{equation}
where $Q_2, \Td Q_2\in L^2_{{\rm cl\,}}(M)$. From \eqref{e-gue150109I} and \eqref{e-gue150109II}, we conclude that 
$\Upsilon\colon H^s(M)\To H^{s+\frac{1}{2}}(M)$ for all $s\in\mathbb Z$ and $\Upsilon$ is a smoothing operator of order $1$. Fix $K\in\mathbb N$. Put 
\[B_K:=(\tau H\tau)(\tau+\Upsilon+\Upsilon^2+\cdots+\Upsilon^K).\]
Then, $B_K$ is a smoothing operator of order $4-\varepsilon$ for all $0<\varepsilon<1$. 
From \eqref{e-gue150109}, we have that
\[B_k\oP^\prime_4=\tau-\Upsilon^{K+1}.\]
Since $\Upsilon^{K+1}\colon H^s(M)\To H^{s+\frac{K+1}{2}}(M)$ for every $s\in\mathbb Z$, given $\ell\in\mathbb N_0$, we can take $K$ large enough so that $\Upsilon^{K+1}(x,y)\in C^\ell(M\times M)$. The theorem follows.
\end{proof}

In the proof of Theorem~\ref{thm:regularity}, we also need the following result.

\begin{thm}\label{t-gue150109}
Let $w\in L^2(M)$. If $\Delta_bw\in L^2(M)$, then there is a constant $c>0$ such that $e^{c\abs{w}^2}\in L^1(M)$.
\end{thm} 

To prove Theorem~\ref{t-gue150109}, we need the following Adams-type theorem of Fontana and Morpurgo~\cite{FontanaMorpurgo2011}.

\begin{thm}\label{t-gue150109I}
Let $A\colon L^2(M)\To L^2(M)$ be a continuous operator with distribution kernel $A(x,y)\in C^\infty(M\times M\setminus{\rm diag\,}(M\times M))$. Suppose that the kernel $A(x,y)$ satisfies
\begin{equation}\label{e-gue150109a}
\begin{split}
&\sup_{x\in M}\abs{\set{y\in M\colon \abs{A(x,y)}>s}}\leq Ks^{-2},\\
&\sup_{y\in M}\abs{\set{x\in M\colon \abs{A(x,y)}>s}}\leq Ks^{-2}
\end{split}
\end{equation}
as $s\To\infty$, where $K>0$ is a constant and 
\[\abs{\set{y\in M\colon \abs{A(x,y)}>s}},\ \ \abs{\set{x\in M\colon \abs{A(x,y)}>s}}\]
denote the volumes of the sets $\set{y\in M\colon \abs{A(x,y)}>s}$ and $\set{x\in M\colon \abs{A(x,y)}>s}$, respectively with respect to the given volume form on $M$. Then, for any $f\in L^2(M)$ with $Tf\in L^2(M)$, there is a constant $c>0$ such that $e^{c\abs{f}^2}\in L^1(M)$.
\end{thm}

\begin{proof}[Proof of Theorem~\ref{t-gue150109}]
Put $g:=(\Delta_b+I)w\in L^2(M)$. Let $Q$ be the inverse of $\Delta_b+I$. Then, $w=Qg$. It is known that (see~\cite[Section 2]{ChengMalchiodiYang2013} and~\cite[Section 10]{HsiaoYung2014})
\begin{equation}\label{e-gue150110}
\abs{Q(x,y)}\lesssim\vartheta(x,y)^{-2}. 
\end{equation}
From \eqref{e-gue150110}, one readily checks that
\begin{equation}\label{e-gue150110a}
\begin{split}
&{\rm sup\,}_{x\in M}\abs{\set{y\in M;\, \abs{Q(x,y)}>s}}\lesssim s^{-2},\\
&{\rm sup\,}_{y\in M}\abs{\set{x\in M;\, \abs{Q(x,y)}>s}}\lesssim s^{-2},
\end{split}
\end{equation}
as $s\To\infty$. The conclusion follows from~\eqref{e-gue150110a} and Theorem~\ref{t-gue150109I}.
\end{proof}

\end{document}